\documentclass{svjour2}                    
\smartqed  
\usepackage{graphicx}
\usepackage{amsmath} \usepackage{amsfonts}
\usepackage{amssymb} \usepackage{graphics} \usepackage{graphicx}
\usepackage{amsbsy} \usepackage{url}
\usepackage{amscd} 

\newcommand{\E}{{\mathbb E}} \newcommand{\bC}{{\mathbb C}}

\newcommand{\R}{{\mathbb R }}

 \newcommand{\cA}{\mathcal A}
\newcommand{\cL}{\mathcal L}

\newcommand{\Tr}{\mathrm{Tr }}

\begin{document}

\title{Optimal non-reversible linear drift for the convergence to
  equilibrium of a diffusion\thanks{This work is supported by the Agence Nationale de la Recherche, under grant ANR-09-BLAN-0216-01 (MEGAS). This work was initiated while the two last authors had a INRIA-sabbatical semester and year in the INRIA-project team MICMAC. The research of GP is partially supported by the EPSRC under grants No. EP/H034587 and No. EP/J009636/1.}
}

\titlerunning{Optimal non-reversible drift}        

\author{T. Leli\`evre \and
F. Nier  \and G.A. Pavliotis}


\institute{T. Leli\`evre  at
              Universit\'e Paris-Est, Cermics,
              and INRIA, MicMac project team,
              Ecole des ponts, 
              6-8 avenue Blaise Pascal,
              77455 Marne la Vall\'ee Cedex  2, France.
              \email{lelievre@cermics.enpc.fr}           
            \and
F. Nier at IRMAR, Campus de Beaulieu, Universit\'e de Rennes 1, 35042
Rennes, France.
\email{francis.nier@univ-rennes1.fr}
            \and
G.A. Pavliotis at Imperial College London, Department of Mathematics,
South Kensington Campus, London SW7 2AZ, England.
 \email{g.pavliotis@imperial.ac.uk}
}

\date{Received: date / Accepted: date}

\maketitle

\begin{abstract}
We consider non-reversible perturbations of reversible diffusions that do
not alter the invariant distribution and we ask whether there exists
an optimal perturbation such that the rate of convergence to
equilibrium is maximized.  We solve this problem for the case of linear
drift by proving the existence of such optimal perturbations and by
providing an easily implementable algorithm for constructing them. We
discuss in particular the role of the prefactor in the exponential
convergence estimate. Our
rigorous results are illustrated by numerical experiments.
\end{abstract}

\section{Introduction}
\label{sec:intro}

\subsection{Motivation}

The problem of convergence to equilibrium for diffusion processes has
attracted considerable attention in recent years. In addition to the
relevance of this problem for the convergence to equilibrium
of some systems in statistical
physics, see for example~\cite{Marko_Vill00}, such questions are also important in
statistics, for example in the analysis of Markov Chain Monte Carlo
(MCMC) algorithms~\cite{Diaconis2009}. Roughly speaking, one measure of efficiency of an
MCMC algorithm is its rate of convergence to equilibrium, and
increasing this rate is thus the aim of many numerical techniques (see
for example~\cite{chopin-lelievre-stoltz-12}).

Let us recall the basic approach for a reversible diffusion. Suppose
that 
we are interested in sampling from a probability distribution function
\begin{equation}\label{eq:psi_infty}
\psi_{\infty}= \frac{e^{-V}}{\int_{\R^N} e^{-V}~dx}\,,
\end{equation}
where $V: \R^N \to \R$ is a given smooth potential such that
$\int_{\R^N} e^{-V}~dx< \infty$\,. A natural dynamics to use is the
reversible dynamics
\begin{equation}\label{e:reversible} dX_{t} = -\nabla V(X_{t}) \, dt +
\sqrt{2} \, dW_{t}\,,
\end{equation}
where $W_t$ denotes a standard $N$-dimensional Brownian motion.
Let us denote by $\psi_{t}$ the probability density function of the process $X_{t}$
at time~$t$. It satisfies the Fokker-Planck equation
\begin{equation}\label{e:FP}
\partial_{t} \psi_{t} = \nabla \cdot \left(\nabla V \psi_{t} + \nabla
\psi_{t} \right)\,.
\end{equation}
Under appropriate assumptions on the potential $V$
(e.g. that $\frac{1}{2} |\nabla V(x)|^{2} - \Delta V(x) \rightarrow
+\infty$ as $|x| \rightarrow +\infty$\,, see~\cite[A.19]{villani-09}),
the density $\psi_\infty$ satisfies a Poincar\'e inequality: there
exists $\lambda >0$ such that for all probability density functions $\phi$,
\begin{equation}\label{eq:poincare}
\int_{\R^N} \left(\frac{\phi}{\psi_\infty} -1 \right)^2 \psi_\infty dx
\le 
\frac{1}{\lambda} \int_{\R^N} \left|\nabla \left(
    \frac{\phi}{\psi_\infty}\right) \right|^2 
\psi_\infty dx\,.
\end{equation}
The optimal parameter $\lambda$ in~\eqref{eq:poincare} is the opposite of the
smallest (in absolute value) non-zero eigenvalue of the Fokker-Planck operator $\nabla \cdot \left(\nabla V \cdot + \nabla
\cdot \right)$, which is self-adjoint in $L^{2}(\R^{N} ,
\psi_{\infty}^{-1} \, dx)$ (see~\eqref{eq:FP_rev} below). Thus, $\lambda$ is also called the spectral gap of the Fokker-Planck operator.

It is then standard to show that~\eqref{eq:poincare} is equivalent to the following inequality, which shows exponential convergence to the equilibrium  for~\eqref{e:reversible}: for all initial conditions $\psi_0 \in L^{2}(\R^{N},\psi_{\infty}^{-1} dx)$, for all times $t \ge 0$,
\begin{equation}
\label{e:converg} 
\|\psi_{t} - \psi_{\infty} \|_{L^2(\psi_\infty^{-1})} \leq
e^{- \lambda t}\| \psi_{0} - \psi_{\infty} \|_{L^2(\psi_\infty^{-1} )}\,,
\end{equation}
 where  $\| \cdot \|_{L^2(\psi_\infty^{-1} )}$ denotes the norm in $L^{2}(\R^{N},
\psi_{\infty}^{-1} )$, namely
$\|f\|_{L^2(\psi_\infty^{-1})}^2=\int_{\R^N} f^2(x)
\psi_{\infty}^{-1}(x) \, dx$\,.
This equivalence is a simple consequence of the following identity: if $\psi_t$ is solution to~\eqref{e:FP}, then
\begin{equation}\label{eq:ent_eq}
\frac{d}{dt}  \|\psi_{t} - \psi_{\infty} \|_{L^2(\psi_\infty^{-1})}^2
= -2 \int_{\R^N} \left| \nabla \left(\frac{\psi_t}{\psi_\infty}\right)
\right|^2 \psi_\infty dx\,.
\end{equation}

In view of~\eqref{e:converg}, the algorithm is efficient if $\lambda$
is large, which is typically not the case if $X_t$ is a metastable
process (see~\cite{lelievre-12}). A natural question is therefore how to
design a Markovian dynamics which converges to equilibrium distribution $\psi_\infty$ (much) faster than~\eqref{e:reversible}. There are many approaches (importance sampling methods, constraining techniques, see for example~\cite{lelievre-rousset-stoltz-book-10}), and the focus here is on modifying the dynamics~\eqref{e:reversible} to a non-reversible dynamics, which has the same invariant measure.

\subsection{Non-reversible diffusion}

As noticed in~\cite{Hwang_al1993,Hwang_al2005}, one way to accelerate
the convergence to equilibrium is to depart from reversible
dynamics (see also~\cite{diaconis-miclo-12} for related discussions for Markov Chains). Let us recall that the dynamics~\eqref{e:reversible} is
reversible in the sense that if $X_0$ is distributed according to
$\psi_\infty(x) \, dx$, then $(X_t)_{0 \le t \le T}$ and $(X_{T-t})_{0 \le t \le T}$ have the same law. This is equivalent to the fact that the Fokker-Planck operator is self-adjoint in $L^{2}(\R^{N},\psi_{\infty}^{-1} dx)$: 
\begin{equation}\label{eq:FP_rev}
\begin{aligned}
\int_{\R^N} \nabla \cdot \left(\nabla V \psi + \nabla
\psi \right) \phi \, \psi_\infty^{-1}dx&= -\int_{\R^N} \nabla
\left(\psi \psi_\infty^{-1} \right) \cdot \nabla \left( \phi
  \psi_\infty^{-1}\right) \psi_\infty dx\\
& =\int_{\R^N} \nabla \cdot \left(\nabla V \phi + \nabla
\phi \right) \psi \, \psi_\infty^{-1}dx\,.
\end{aligned}
\end{equation}
Now, a natural {\em non-reversible} dynamics to sample from the
distribution $\psi_\infty(x) \, dx=\frac{e^{-V(x)}}{\int_{\R^N} e^{-V}~dx}\, dx$ is:
\begin{equation}\label{e:nonreversible} dX^{b}_{t} = \big(-\nabla
V(X^{b}_{t}) + b(X^{b}_{t}) \big) \, dt + \sqrt{2} \, dW_{t},
\end{equation} 
where $b$ is taken to be divergence-free with respect to the invariant
distribution 
$\psi_{\infty}(x)\,dx$:
\begin{equation}\label{e:divergence_free} \nabla \cdot \big( b e^{-V}
\big) = 0,
\end{equation}
 so that $\psi_\infty(x) \, dx $ is still the invariant measure of the dynamics~\eqref{e:nonreversible}. A general way to construct such a $b$ is to consider 
\begin{equation}\label{e:perturb-gen}
b=J \nabla V,
\end{equation} 
where $J$ is a constant antisymmetric matrix.

It is important to note that the
dynamics~\eqref{e:nonreversible} is non-reversible. Indeed, one can
check that $(X^b_t)_{0 \le t \le T}$ has the same law as
$(X^{-b}_{T-t})_{0 \le t \le T}$ (notice the minus sign in front of
$b$), and thus not the same law as $(X^{b}_{T-t})_{0 \le t \le T}$. Likewise, Equation~\eqref{eq:FP_rev} now becomes:
\begin{equation}\label{eq:FP_nonrev}
\int_{\R^N} \nabla \cdot \left((\nabla V-b) \psi + \nabla
\psi \right) \phi \, \psi_\infty^{-1}dx
=\int_{\R^N} \nabla \cdot \left((\nabla V+b) \phi + \nabla
\phi \right) \psi \, \psi_\infty^{-1}dx\,.
\end{equation}
Again, notice the change of sign in front of $b$.

From~\eqref{e:perturb-gen} it is clear that there are many (in fact, infinitely many) different ways to modify the reversible dynamics without changing the invariant measure. A natural question
is whether the addition of a non-reversible term can improve
the rate of convergence to equilibrium and, if so, whether there
exists an optimal choice for the perturbation that maximizes the rate
of convergence to equilibrium. The goal of this paper is to present a complete solution to this problem when the drift term
in~\eqref{e:nonreversible} is linear.

More precisely, let $\psi^{b}_{t}$ denote the law of the process $X_{t}^{b}$, {\em i.e.} the
solution to the Fokker-Planck equation
\begin{equation}\label{e:FPnonsymm}
\partial_{t} \psi^b_{t} = \nabla \cdot \left((\nabla V - b) \psi^b_{t} +
\nabla \psi^b_{t} \right).
\end{equation}
Using the fact that $\psi_\infty$ is a stationary solution
to~\eqref{e:FPnonsymm} (which is equivalent
to~\eqref{e:divergence_free}) and under the assumption that
$\psi_\infty$ satisfies the Poincar\'e inequality~\eqref{eq:poincare},
one can check that the upper bound for the reversible
dynamics~\eqref{e:reversible} is still valid:
\begin{equation}\label{e:converg_nonsymm} \|\psi^b_{t} - \psi_{\infty} \|_{L^2(\psi_\infty^{-1})} \leq
e^{- \lambda t}\| \psi^b_{0} - \psi_{\infty} \|_{L^2(\psi_\infty^{-1})}\,.
\end{equation} 
Actually, as in the reversible case,~\eqref{e:converg_nonsymm} (for
all initial conditions $\psi^b_{0}$) is equivalent to~\eqref{eq:poincare}.
This is because~\eqref{eq:ent_eq} also holds for $\psi^b$ solution to~\eqref{e:FPnonsymm}. In other words, adding a non-reversible part to the dynamics cannot be worse than the original dynamics~\eqref{e:reversible} (where $b=0$) in terms of exponential rate of convergence.

What we show below (for a linear drift) is that it is possible to
choose $b$ in order to obtain a convergence at exponential rate of the form:
\begin{equation}\label{e:converg_exp_nonsymm} 
\|\psi_{t}^{b} -
\psi_{\infty} \|_{L^2(\psi_\infty^{-1})} \leq C(V, \,b) e^{- \overline{\lambda} t}\| \psi_{0}^{b} -
\psi_{\infty} \|_{L^2(\psi_\infty^{-1})}\,,
\end{equation}
with $\overline{\lambda} > \lambda$ and $C(V,b)>1$.
It is important to note the presence of the constant
$C(V,b)$ in the right-hand side
of~\eqref{e:converg_exp_nonsymm}. For a reversible diffusion ($b=0$),
the spectral theorem forces the optimal $C(V,0)$ to be equal to one,
and $\overline{\lambda}=\lambda$, the Poincar\'e inequality constant of
$\psi_\infty$ (since~\eqref{e:converg}
implies~\eqref{eq:poincare}). The interest in adding a non-reversible
perturbation is precisely to allow for a constant $C(V,b)>1$\,, which
permits a rate $\overline{\lambda} > \lambda$\,. 
The difficulty is thus to design a $b$ such that $\overline{\lambda}$
is large and $C(V,b)$ is not too large. In the following, we adapt a two-stage strategy: we first optimize $b$ in order to get the largest possible $\overline{\lambda}$, and then we discuss how the constant $C(V,b)$ behaves for this optimal rate of convergence.


\subsection{Bibliography}

This problem was studied in~\cite{Hwang_al1993} for a linear drift (namely $V$ is quadratic and $b$ is linear)
and in~\cite{Hwang_al2005} for the general case. It was shown in these
works that the addition of a drift function $b$
satisfying~\eqref{e:divergence_free} helps to speed up convergence to
 equilibrium. Furthermore, the optimal convergence rate was obtained
for the linear problem (see also Proposition~\ref{prop:spectrumB} in
the present paper) and some explicit examples were presented, for ordinary differential equations in two and three dimensions.

The behavior of the generator of the dynamics~\eqref{e:nonreversible}
under a strong non-reversible drift has also been
studied~\cite{Berestycki_Hamel_Nadirashvili,ConstKiselRyzhZl06,Franke_al2010}. It was shown
in~\cite{Franke_al2010} that the spectral gap attains a finite value
in the limit as the strength of the perturbation becomes infinite if
and only if the operator $b \cdot \nabla$ has no eigenfunctions in an appropriate
Sobolev space of index $1$. These works, although relevant to our work, are not directly
related to the present paper since our main focus is in obtaining the
optimal perturbation rather than an asymptotic result. The effect of non-reversible perturbations to the constant in logarithmic Sobolev inequalities (LSI) for diffusions have also been studied, see~\cite{ArnoldCarlenJu2008}. In this paper examples were presented where the addition of a non-reversible perturbation can improve the constant in the LSI. 

This work is also related to~\cite{girolami-calderhead-11}, where the
authors use another idea to enhance the convergence to
equilibrium. The principle is to keep a reversible diffusion, but to change
the underlying Riemannian metric by considering $$dX^M_t=- D \nabla V
(X^M_t) \, dt +\sqrt{2D} dW_t$$
for a well chosen matrix $D$. More precisely, the authors apply this
technique to a Hybrid Monte Carlo scheme. It would be interesting to set up some test cases in order to
compare the two approaches: non-reversible drift {\em versus} change
of the underlying metric.

Finally, we would like to mention related recent works on spectral
properties of non-selfadjoint operators~see for example
\cite{Davies,GGN,Trefethen_Embree} and references therein.

\subsection{Outline of the paper}

In this paper, we study the case of a linear drift. Namely, we
consider~\eqref{e:reversible} with a quadratic potential
\begin{equation}\label{e:quadratic} V(x) =\frac{1}{2} x^{T} S x,
\end{equation} where $S$ is a positive definite $N \times N$
symmetric matrix. In the following, we denote $\mathcal{S}_{N}(\R)$ the set of
symmetric matrices and $\mathcal{S}_{N}^{>0}(\R)$ the set
of positive definite symmetric matrices. The equilibrium distribution
thus has the density
\begin{equation}\label{eq:psi_infty_lin}
\psi_\infty(x)=\frac{\det(S)^{1/2}}{(2\pi)^{N/2}} \exp\left(-\frac{x^TSx}{2}\right)\,.
\end{equation}
It can be checked that if the vector field $b(x)$ is linear, it satisfies~\eqref{e:divergence_free} if and only if $b =
-J S x$ with $J= -J^{T}$ an antisymmetric real matrix, see Lemma~\ref{lem:nonrev_quad}. For a given $S$, the question is thus how to choose $J$ in order to optimize the rate of convergence to equilibrium for the dynamics~\eqref{e:nonreversible}, which in our case becomes:
\begin{equation}\label{e:nonrev_lin}
dX^J_t=-(I+J)S X^J_t \, dt + \sqrt{2}\, dW_t,
\end{equation}
where $I$ denotes the identity matrix in $\mathcal{M}_{N}(\R)$, the set of $N \times N $ real valued matrices.

We provide an answer to this question. In particular:
\begin{enumerate}
\item We prove that it is possible to build an optimal $J$ (denoted $J_{opt}$), which yields the best possible rate $\overline{\lambda}$ (denoted $\lambda_{opt}$) in~\eqref{e:converg_exp_nonsymm}.
\item We provide an algorithm for constructing an optimal matrix $J_{opt}$.
\item We obtain estimates on the constant $C(V, \,b) = C (S, \, J)$
in~\eqref{e:converg_exp_nonsymm}.
\end{enumerate}
It appears that this procedure becomes particularly relevant in the
situation when the condition number of $S$ is large (namely for
an original dynamics with multiple timescales, see
Sections~\ref{sec:2D} and~\ref{sec:num}). Discussions about the
size of  $C(V,b)$
with respect to this conditioning and to the dimension $N$ can be
carried out very accurately.

The reason why the case of linear drift is amenable to
analysis is because it can be reduced to a linear
algebraic problem, at least for the calculation of $\lambda_{opt}$ and the
construction of $J_{opt}$. One way to understand this is the following remark: the spectrum of an operator of the form
(which is precisely the form of the generator of the
dynamics~\eqref{e:nonrev_lin})
\begin{equation}\label{e:generator} 
\cL =  -(B x) \cdot \nabla  + \Delta\,,
\end{equation}
can be computed in terms of the eigenvalues of the
matrix $B$. Here, $B$ denotes any real square matrix with positive
real spectrum. In~\cite{Metafune_al2002} (see also~\cite{ottobre_pavliotis_pravda-starov,ottobre_pavliotis_pravda-starov_preprint} and
Proposition~\ref{prop:spectLJ} below), it was indeed proven that the spectrum
of $\cL$ in $L^{p}$ spaces weighted by the invariant measure
of the dynamics ($p > 1$) consists of integer linear combinations of eigenvalues
of $B$:
\begin{equation}\label{e:sprectrum} \sigma (\cL) =
\left\{-\sum_{j=1}^{r} n_{j} \lambda_{j}, \; n_{j} \subset \mathbb{N}
\right\},
\end{equation} where $\left\{ \lambda_{j} \right\}_{j=1}^{r}$ denote
the $r$ (distinct) eigenvalues of $B$. In particular, the spectral gap
of the generator $\cL$ is determined by the eigenvalues of $B$, and this yields a simple way to design the optimal matrix $J_{opt}$. On the other hand, the
control of the constant $C(S,  J)$ requires a more elaborate
analysis, using Wick (in the sense of Wick ordered) calculus, see Section~\ref{sec:returnLJ} below.

Compared to the related previous paper~\cite{Hwang_al1993}, our
contributions are threefold: (i) we propose an algorithm to build the
optimal matrix $J_{opt}$, (ii) we discuss how to get estimates on the
constant $C(S,J)$ and (iii) we consider the longtime behavior of the
partial differential equation~\eqref{e:FPnonsymm} and not only of
ordinary differential equations related to~\eqref{e:nonrev_lin}. In
particular, our analysis covers also non Gaussian initial conditions
for the SDE~\eqref{e:nonrev_lin}. Although the results that we obtain have
a limited practical interest (there exist many efficient techniques
to draw Gaussian random variables), we believe that this study is a
first step towards further analysis, in particular for nonlinear drift terms.

The rest of the paper is organized as follows. In
Section~\ref{sec:main} we present the main results of this paper. In
Section~\ref{sec:resc} we perform some preliminary calculations. The linear algebraic problem
and the evolution of the corresponding ordinary differential equation
are studied in Section~\ref{sec:linalg}. Direct computations of the
expectations and the variances are performed in
Section~\ref{sec:gauss} for Gaussian initial data. The convergence to
 equilibrium for the non-reversible diffusion process for general
 initial data is then studied in
Section~\ref{sec:diffusion}. Results of numerical simulations are
presented in Section~\ref{sec:num}. Finally, some background material on
Wick calculus, which is needed in the proofs of our main results, is
presented in Appendix~\ref{sec:wickcalculus}.
%
%
%
%
\subsection{Main results}
\label{sec:main}
For a potential given by~\eqref{e:quadratic}, our first result is a
simple lemma which characterizes  all non-reversible perturbations that
satisfy divergence-free condition~\eqref{e:divergence_free}.
\begin{lemma}\label{lem:nonrev_quad} Let $V(x)$ be given
by~\eqref{e:quadratic} and let $b(x) = -A x$ where $A \in {\mathcal M}_N(\R)$. Then~\eqref{e:divergence_free} is satisfied if and only if
\begin{equation}\label{e:div_free_quad} A = J S\,, \quad\mbox{with} \;\;
J = - J^{T}\,.
\end{equation}
\end{lemma}
\begin{proof} Equation~\eqref{e:divergence_free} with $b = - A x$ and
quadratic potential~\eqref{e:quadratic} gives
$$
\nabla \cdot \left(A x \, e^{-\frac{x^{T} S x}{2}} \right) = 0
$$
which is equivalent to
$$
\Tr (A) +  (Ax)^T(S x) = 0 \quad \forall x \in
\R^{N}\,.
$$
This is equivalent to the conditions
$$
\Tr(A) = 0 \quad \text{and} \quad (A^{T} S) = - (A^{T} S)^{T}\,.
$$
Set now $J = A S^{-1}$\,. We have
$$
\Tr (J S) = 0 \quad \text{and} \quad S (J +J^{T}) S = 0\,,
$$
which is equivalent to
$$
J = - J^{T}\,.
$$
\end{proof}

We will denote the set of $N\times N$ real antisymmetric matrices by
$\cA_{N}(\R)\subset\mathcal{M}_{N}(\R)$\,. The following result concerns the optimization of the spectrum of the matrix $B_{J}=(I+J)S$, which appears in the drift of the dynamics~\eqref{e:nonrev_lin} and plays a crucial role in the analysis; see Equation~\eqref{e:sprectrum}.
\begin{theorem}\label{thm:lin_alg}
 Define $B_{J} = (I + J) S$\,. Then
\begin{equation}\label{e:minmax} \max_{J \in \cA_{N}(\R)} \min {\rm Re}
\left(\sigma (B_{J}) \right) = \frac{\Tr (S)}{N}\,.
\end{equation}
Furthermore, one can construct matrices $J_{opt} \in
\cA_{N}(\R)$ such that the maximum in~\eqref{e:minmax} is attained.
The matrix $J_{opt}$ can be chosen so that
 the semi-group associated to $B_{J_{opt}}$ satisfies the bound
\begin{equation}\label{eq:lin_alg}
\left\|e^{-(I+J_{opt})S t} \right\| \leq C_{N}^{(1)} \kappa(S)^{1/2} \exp
\left(-\frac{\Tr (S)}{N} t\right)\,,
\end{equation}
where  the matricial norm is induced by
the euclidean norm on $\R^{N}$ and $\kappa (S)=\|S\|\, \|S^{-1}\|$ denotes the condition number.
\end{theorem}
Theorem~\ref{thm:lin_alg} is a straightforward consequence of Proposition~\ref{pr:opt}
and Proposition~\ref{prop:conv_ode} below, with an explicit expression
for the constant $C_N^{(1)}$ given by~\eqref{eq:CN_ODE}. This
expression allows to  discuss the dependence of $C_N^{(1)}$ on the dimension~$N$ (see
Remark~\ref{rem:CN_ODE} for details).

The partial differential equation version of this result requires to introduce the generator
$$
\mathcal{L}_{J}=-(B_{J}x) \cdot \nabla +\Delta
$$
of the semigroup $(e^{t\mathcal{L}_{J}})_{t\geq 0}$ considered in $L^{2}(\R^{N},
\psi_\infty dx ; \bC)$, where, we recall (see~\eqref{eq:psi_infty_lin}),
 $$\psi_\infty(x)=\frac{\det(S)^{1/2}}{(2\pi)^{N/2}} \exp\left(-\frac{x^TSx}{2}\right).$$ Here $L^{2}(\R^{N},
\psi_\infty dx ; \bC)$ denotes the set of functions $f:\R^N \to \bC$ such that $\int_{\R^N} |f|^2(x) \psi_\infty(x) \, dx < \infty$.
\begin{theorem}\label{th:CV_KB}
For $B_{J}=(I+J)S$ with $J\in \mathcal{A}_{N}$\,, the drift-diffusion operator
$\mathcal{L}_{J}=-(B_{J}x).\nabla +\Delta$ defined in $L^{2}(\R^{N},
\psi_\infty dx;\bC)$ with domain of definition 
$$
D(\mathcal{L}_{J})=\left\{u\in L^{2}(\R^{N},
\psi_\infty dx ; \bC)\,,~ \mathcal{L}_{J}u\in L^{2}(\R^{N},
\psi_\infty dx ; \bC)
\right\}
$$
generates a contraction semigroup $(e^{t\mathcal{L}_{J}})_{t\geq
  0}$ and it has a compact resolvent. Optimizing its spectrum with respect to $J$  gives
\begin{equation}
\label{e:minmaxop} \max_{J \in \cA_{N}(\R)} \min {\rm Re}
\left(\sigma (-\mathcal{L}_{J}) \setminus \{0\} \right) = \frac{\Tr (S)}{N}\,.
\end{equation}
Furthermore, the maximum in~\eqref{e:minmaxop} is attained for the matrices $J_{opt}\in \mathcal{A}_{N}(\R)$ constructed
as in Theorem~\ref{thm:lin_alg}. The matrix $J_{opt}$ can be chosen
so that
\begin{equation}\label{eq:CV_KB}
\begin{aligned}
&\left\|e^{t\mathcal{L}_{J_{opt}}}u -\left(\int_{\R^N} u
\psi_\infty dx\right)\right\|_{L^{2}(\psi_\infty )} \\
& \quad \leq C_{N}^{(2)}\kappa(S)^{7/2}\exp
\left(-\frac{\Tr (S)}{N} t\right) \left\|u -\left(\int_{\R^N} u
\psi_\infty dx\right)\right\|_{L^{2}(\psi_\infty )}
\end{aligned}
\end{equation}
holds for all $ u\in L^{2}(\R^{N},\psi_\infty dx ; \bC)$ and all $t\geq 0$, 
where $\kappa (\cdot)$ again denotes the condition number.
\end{theorem}
Theorem~\ref{th:CV_KB} is a straightforward consequence of
Proposition~\ref{prop:return1} below, with an explicit expression
for the constant $C_N^{(2)}$ given by~\eqref{eq:CN_PDE}. 
Again, the dependence of $C_N^{(2)}$ on the dimension $N$ is discussed in Remark~\ref{rem:CN_PDE}. A simple corollary of this result is the following:
\begin{corollary}\label{cor:CV_FP}
Let us consider the Fokker Planck equation associated to the dynamics~\eqref{e:nonrev_lin} on $X^J_t$:
\begin{equation}\label{eq:FP_lin}
\partial_t \psi^J_t = \nabla \cdot \left(B_J x \, \psi^J_t + \nabla \psi^J_t\right),
\end{equation}
where  $B_{J} = (I + J) S$. Let us assume that $\psi^J_0 \in
L^2(\R^N,\psi_\infty^{-1} \, dx)$. Then, by considering $J=-J_{opt}$,
where $J_{opt}\in \cA_N(\R)$ refers to the matrix considered in
Theorem~\ref{th:CV_KB} to get~\eqref{eq:CV_KB}.
Then the inequality 
$$
\left\|\psi^J_t - \psi_\infty \right\|_{L^{2}(\psi_\infty^{-1})} \leq C_{N}^{(2)}\kappa(S)^{7/2}\exp
\left(-\frac{\Tr (S)}{N} t\right) \left\|\psi^J_0 - \psi_\infty \right\|_{L^{2}(\psi_\infty^{-1})} \,,
$$
holds for all $t \ge 0$\,,
when $\psi_\infty$ is defined by~\eqref{eq:psi_infty_lin}.
\end{corollary}
\begin{proof}
This result is based on the following simple remark: $\psi^J_t$ is a solution to~\eqref{eq:FP_lin} in $L^2(\R^N,\psi_\infty^{-1} \, dx)$ if and only if $\psi^J_t \, \psi_\infty^{-1} = e^{t {\mathcal L}_{-J}} ( \psi^J_0 \, \psi_\infty^{-1})$ in $L^2(\R^N,\psi_\infty \, dx)$. Notice the minus sign in ${\mathcal L}_{-J}$. Then the exponential convergence is obtained from~\eqref{eq:CV_KB} using the equality:
$$
\left\|\psi^J_t - \psi_\infty
\right\|_{L^{2}(\psi_\infty^{-1})}=\left\|\psi^J_t \,
  \psi_\infty^{-1} - \left( \int_{\R^N} \psi^J_0 \psi_\infty^{-1}
    \psi_\infty dx\right) \right\|_{L^{2}(\psi_\infty )}\,.
$$
\end{proof}
\begin{remark}
\label{rem:moregen} A more general result but with a less accurate
upper bound is given in Proposition~\ref{pr.gauss}.
\end{remark}
\begin{remark}
The partial differential equation
$$\partial_t f = {\mathcal L}_J f = - (B_J x) \cdot \nabla f + \Delta f$$
which we consider in Theorem~\ref{th:CV_KB} is sometimes called the backward Kolmogorov equation associated with the dynamics~\eqref{e:nonrev_lin}. It is related to this stochastic differential equation through the Feynman-Kac formula:
$$f(t,x)=\E^x f(X_t^J)$$
where $X^J_t$ is the solution to~\eqref{e:nonrev_lin} and $\E^x$ indicates that we consider a solution starting from $x \in \R^N$: $X^J_0=x$. The partial differential equation
$$\partial_t \psi^J_t = \nabla \cdot ( B_J x \, \psi^J_t + \nabla \psi^J_t)$$
which we consider in the Corollary~\ref{cor:CV_FP} is the Fokker
Planck (or forward Kolmogorov) equation associated
with~\eqref{e:nonrev_lin}: if $X_0^J \sim \psi^J_0(x) \, dx$, then  for all times $t > 0$, $\psi^J_t$ is the probability density function of $X^J_t$.

As explained in the proof of Corollary~\ref{cor:CV_FP} above, these two partial
differential equations are related through a conjugation. See also,
e.g.~\cite{villani-09,Oks98}. 

\end{remark}

\begin{remark}
It would be interesting to explore extensions of this approach to the Langevin
dynamics:
$$
\left\{
\begin{aligned}
dq_t&=p_t \, dt,\\
dp_t&= - \nabla V (q_t) \, dt - \gamma p_t \, dt + \sqrt{2 \gamma} \, dW_t,
\end{aligned}
\right.
$$
which is ergodic with respect to the measure $Z^{-1} \exp(-V(q)
-|p|^2/2) \, dp dq$. For example the following modification
$$
\left\{
\begin{aligned}
dq_t&= (I - J) p_t \, dt,\\
dp_t&= - ( I + J) \nabla V (q_t) \, dt - \gamma p_t \, dt + \sqrt{2 \gamma} \, dW_t,
\end{aligned}
\right.
$$
where $J$ is an antisymmetric matrix leaves the measure $Z^{-1} \exp(-V(q)
-|p|^2/2) \, dp dq$ stationary. In the linear case $V(x)=\frac{x^T S x}{2}$, this leads to a Kramers-Fokker-Planck operator which is a
differential operator  (at most) quadratic in
$(q,p,\partial_{q}, \partial_{p})$\,.
Then the exponential decay rate can be reduced to some (more involved)
linear algebra problem following \cite{HiPr}. About the constant
prefactor in front of the decaying in time exponential, the argument
based on
sectiorality used in Lemma~\ref{lem:sector} does not apply anymore. It
has to be replaced by hypoelliptic estimates in the spirit of \cite{EcHa,HeNi04,HiPr}.
The reference \cite{HiPr} provides accurate results for differential
operators with at most quadratic symbols. 
\end{remark}

\begin{remark}
We notice that the fundamental property  ${\rm div}(b e^{-V})=0$ is
still satisfied for $b(t,x)=J(t)\nabla V(x)$, where $J(t)$ is a
time-dependent (deterministic) antisymmetric matrix. This could be useful for further
generalization of this approach.
\end{remark}

%
%



\section{A useful rescaling}
\label{sec:resc}

The analysis will be carried out in a suitable system of coordinates
which simplifies the calculations and the presentation of the
intermediate results.  We will
perform one conjugation and a change of variables.

First, from the partial differential equation point of view, it
appears to be useful to work in $L^{2}(\R^{N}, dx;\bC)$ instead of
$L^{2}(\R^{N}, \psi_\infty dx;\bC)$, since this allows to use standard
techniques for the spectral analysis of partial differential
equations. In the following, the norm in  $L^{2}(\R^{N}, dx;\bC)$ is
simply denoted $\| \cdot \|_{L^2}$\,.
For a general potential $V$, the mapping $u\mapsto
\psi_\infty^{-1/2} u$ sends unitarily
$L^{2}(\R^{N}, dx;\bC)$ into $L^{2}(\R^{N}, \psi_\infty dx;\bC)$ with
the associated transformation rules for the differential operators:
\begin{eqnarray*} &&e^{-\frac{V}{2}}\nabla e^{\frac{V}{2}}=\nabla
+\frac{1}{2}\nabla V\quad,\quad e^{-\frac{V}{2}}\nabla^{T}
e^{\frac{V}{2}}=\nabla^{T} +\frac{1}{2}\nabla V^{T}\,,\\ && \nabla=
\begin{pmatrix}
  \partial_{x_{1}}\\ \vdots\\
\partial_{x_{N}}
\end{pmatrix} \quad,\quad
\nabla^{T}=(\partial_{x_{1}},\ldots, \partial_{x_{N}})\,,\quad
(\nabla^{T}X=\mbox{div\,} X)\,.
\end{eqnarray*}
Thus, the operator
$$
\mathcal{L}=-\nabla V^{T}\nabla + b^{T}\nabla + \Delta
$$ 
is transformed into
\begin{equation}\label{e:schr-b}
\overline{\mathcal{L}}=e^{-\frac{V}{2}}\mathcal{L}e^{\frac{V}{2}} = 
\Delta -\frac{1}{4}|\nabla V|^{2}+\frac{1}{2}\Delta V
+b^{T}\nabla + \frac{1}{2}b^{T}\nabla V\,.
\end{equation}
In the linear case we consider in this paper, $V(x)=\frac{1}{2}x^{T}Sx$
(where $S=S^{T}$ is positive definite), $b(x)=-Ax$ and $A=JS$, $J\in \mathcal{A}_{N}(\R)$,
(see Lemma~\ref{lem:nonrev_quad}), so that the operator
$$
\mathcal{L}=\mathcal{L}_{J}=-(B_{J}x)^{T}\nabla+\Delta\quad\text{with}\quad B_{J}=(I+J)S\,,
$$
becomes
\begin{align*} 
\overline{\mathcal{L}_{J}}
&=
\Delta-\frac{1}{4}x^{T}S^{2}x + \frac{1}{2}\Tr(S)
-x^{T}A^{T}\nabla -\frac{1}{2}x^{T}A^{T}Sx\\ 
&= \Delta-\frac{1}{4}x^{T}S^{2}x + \frac{1}{2}\Tr(S)
-x^{T}SJ^{T}\nabla -\frac{1}{2}x^{T}SJ^{T}Sx
\\
 &=
\Delta-\frac{1}{4}x^{T}S^{2}x + \frac{1}{2}\Tr(S)
+\frac{1}{2}(x^{T}SJ\nabla -\nabla^{T}JSx)\,.
\end{align*}
For the last line we have used
\begin{eqnarray*}
  &&J^{T}=-J\quad,\quad x^{T}SJ^{T}Sx=0\,,\\
&&
\nabla^{T}Bx=\sum_{i,j}\partial_{x_{i}}B_{ij}x_{j}=\sum_{i,j}x_{j}B_{ij}\partial_{x_{i}}+\sum_{i}B_{ii}=
x^{T}B^{T}\nabla +\Tr(B)\,,\\
\text{with}&& B=SJ^{T}\quad,\quad B^{T}=-JS \quad\text{and}\quad \Tr(SJ)=\Tr\left(S^{1/2}JS^{1/2}\right)=0\,.
\end{eqnarray*}
According to Lemma~\ref{lem:nonrev_quad}, we know that the kernel of
$\overline{\mathcal{L}_{J}}$ is $\bC
e^{-\frac{V}{2}}=\bC e^{-\frac{x^{T}Sx}{4}}$. The operator $\overline{\mathcal{L}}_{J}$ is unitarily equivalent to the operator
 $\mathcal{L}_{J}$.

The aim of the second change of variables is to modify the kernel of the
operator to a centered Gaussian with covariance matrix being the identity. Let us introduce the new coordinates
$$
x=S^{-1/2}y\quad,\quad \nabla_{x}= S^{1/2}\nabla_{y}\,.
$$
Then the operator $\overline{\mathcal{L}_{J}}$ becomes:
\begin{equation} \label{eq:deftLJ} 
\tilde{\mathcal{L}_{J}}=
\nabla_{y}^{T}S\nabla_{y}-\frac{1}{4}y^{T}Sy+\frac{1}{2}\Tr(S)
+\frac{1}{2}(y^T \tilde{J} \nabla_{y}
-\nabla^{T} \tilde{J} y)
\end{equation}
where
$$
\tilde{J}=S^{1/2}JS^{1/2} \in \mathcal{A}_{N}(\R)\,.
$$ 
The corresponding stochastic
process is, in the new coordinate system ($Y_t=S^{1/2} X_t$):
$$
dY_t=-(S+\tilde{J})Y_t dt +\sqrt{2} \, S^{1/2}dW_{t}\,.
$$
The $L^{2}$-normalized element of $\ker \tilde{\mathcal{L}_{J}}$ is now
simply the
standard Gaussian distribution
$$
\frac{1}{(2\pi)^{N/4}}e^{-\frac{|y|^{2}}{4}}\,.
$$
Notice that
$\tilde{\mathcal{L}_{J}}$ is still acting in $L^{2}(\R^{N}, dx;\bC)$\,.

As a summary, $u(t,x)$ satisfies
$$\partial_t u = {\mathcal L}_J u$$
if and only if $v(t,y)=\sqrt{\psi_\infty}(S^{-1/2} y) \,u(t,S^{-1/2} y)$ satisfies
$$
\partial_t v = \tilde{\mathcal L}_J v\,.
$$
We have $u(t,x)=e^{t {\mathcal L}_J } u_0(x)$ and $v(t,y)=e^{t \tilde
  {\mathcal L}_J } v_0(y)$ where $u_0=u(0,\cdot)$ and $v_0=v(0,\cdot)$
are related through $v_0(y)=\sqrt{\psi_\infty}(S^{-1/2}
y) \,u_0(S^{-1/2} y)$. In particular, it is easy to check that for all $t \ge 0$\,,
\begin{equation}\label{eq:uv}
\left\|e^{t\mathcal{L}_{J}} u_0 -\left(\int_{\R^N} u_0
\psi_\infty dx\right)\right\|_{L^{2}(\psi_\infty )} =(\det S)^{-1/4} \left\|
e^{t \tilde{\mathcal{L}}_{J}} (I - \Pi_0) v_0 \right\|_{L^2},
\end{equation}
where 
$$
\displaystyle (\Pi_0(v_0)) (y)=(2 \pi)^{-N/2} \left( \int_{\R^N} v_0 (y) e^{-|y|^2/4} \,
dy \right) e^{-|y|^2/4}
$$
is the $L^2$-orthogonal projection of $v_0$ on the kernel
$\bC e^{-\frac{|y|^{2}}{4}}$ of $\tilde{\mathcal L}_J$. Thus,
proving~\eqref{eq:CV_KB} is equivalent to proving
\begin{equation}\label{eq:CV_KB'}
\left\|e^{t\tilde{\mathcal{L}}_{J}}(I-\Pi_{0}) \right\|_{\mathcal{L}(L^2)} \le
C_{N}^{(2)}\kappa(S)^{7/2}\exp
\left(-\frac{\Tr (S)}{N} t\right),
\end{equation}
where here and in the following we use the standard operator norm
$$\left\|A \right\|_{\mathcal{L}(L^2)}=\sup_{
  u \in {L^2(\R^N)}} \frac{\left\|A u\right\|_{L^2}}{\|u\|_{L^2}}$$
for an arbitrary operator $A$.

In the following, we will often work with $\tilde{\mathcal L}_J$ and
$Y_t$ rather than with ${\mathcal L}_J$ and $X_t$.

%
%
%
\section{The linear algebra problem}
\label{sec:linalg}

The stochastic differential equation~\eqref{e:nonreversible} for 
the linear case (quadratic potential) that we consider becomes
\begin{equation}\label{e:nonrev_quad} 
dX_{t} = -(I+J)S X_{t} \, dt +
\sqrt{2} \, dW_{t}\,,
\end{equation} 
and is associated with the drift matrix
\begin{equation}\label{e:drift_matrix} B_{J}:=(I +J)S\,.
\end{equation} 
With the change of variables given in Section~\ref{sec:resc} 
($Y_t=S^{1/2}X_t$), the stochastic differential equation~\eqref{e:nonrev_quad} becomes
$$
dY_t=-(S+\tilde J)Y_t \, dt +\sqrt{2}\,S^{1/2} \, dW_{t}.
$$
The drift matrix is now
\begin{equation}\label{e:drift_matrix_tilde}
 \tilde{B}_{J}=S^{1/2}B_{J}S^{-1/2} =(S+\tilde{J})\,,
\end{equation}
where, we recall, $\tilde{J}=S^{1/2} J S^{1/2} \in {\mathcal A}_N(\R)$\,.
We first collect basic spectral properties of $\tilde{B}_{J}$ (or equivalently of $B_{J}$) when $J\in\mathcal{A}_{N}(\R)$)
  and then show how
this spectrum can be constructively optimized.

\subsection{Spectrum of $\tilde{B}_{J}$ for a general $J \in \mathcal{A}_{N}(\R)$}
\label{sec:specBJ}

\begin{proposition}\label{prop:spectrumB} For  $\tilde{J}\in
  \mathcal{A}_{N}(\R)$ (or equivalently $J=S^{-1/2} \tilde{J} S^{-1/2}
  \in \mathcal{A}_{N}(\R)$) and $S \in \mathcal{S}_{N}^{>0}(\R)$\,, 
the matrix $\tilde{B}_{J}= S + \tilde{J}$
has the following properties:
\begin{enumerate}
\item[(i)] $\sigma (\tilde{B}_{J}) \subset \Big\{z \in \bC, \, {\rm Re}(z) >0\Big\}$\,.
\item[(ii)] $\Tr(\tilde{B}_{J}) = \Tr (S)$\,.
\item[(iii)] $\min {\rm Re} \big[ \sigma (\tilde{B}_{J}) \big] \leq
\frac{\Tr(S)}{N}$\,.
\end{enumerate}
\end{proposition}
Notice that the properties stated above on $\tilde{B}_{J}$ also hold
 on $B_J$ since $\sigma(\tilde{B}_{J})=\sigma({B}_{J})$ and $\Tr(\tilde{B}_{J})=\Tr({B}_{J})$\,.

\begin{proof} Let $\lambda \in \bC$ be an eigenvalue of $\tilde{B}_{J}$ with
corresponding (non-zero) eigenvector $x_{\lambda} \in \bC^N$:
$$
(S+\tilde{J})x_{\lambda}=\tilde{B}_{J} x_{\lambda} = \lambda x_{\lambda}\,.
$$
Since $S$ is a real matrix, the complex scalar product  with $x_{\lambda}$ gives
$$
\lambda |x_{\lambda}|^{2}=|S^{1/2} x_{\lambda}|^{2} + (
x_{\lambda}\,,\,\tilde{J}x_{\lambda} )_{\bC}\,.
$$
Here and in the following, the complex scalar product is taken to be right-linear and
left-antilinear: for any $X$ and $Y$ in $\bC^N$\,,
$$
\left(X,Y\right)_{\bC}=\overline{X}^{T}Y\,.
$$
Using the fact that $\tilde{J} \in \mathcal{A}_N(\R)$\,, we get:
$$
{\rm Re}(\lambda) = \frac{|S^{1/2}x_{\lambda}|^{2}}{|x_{\lambda}|^{2}} >0\,.
$$ 
This ends the proof of $(i)$. The proof of $(ii)$ follows immediately from the fact that the trace of
the antisymmetric  matrix $\tilde{J}$ is $0$\,.

To prove $(iii)$, let
$$
\sigma (\tilde{B}_{J}) = \left\{\lambda_{1}, \lambda_{2}, \ldots, \lambda_{r}
\right\}
$$
denote the spectrum of $\tilde{B}_{J}$\,, and let $m_{k}$ denote the algebraic multiplicity of
$\lambda_{k}$\,. Part $(ii)$ says
$$
\sum_{k=1}^{r} m_{k} \lambda_{k} = \Tr (S) \in \R\,,
$$
and consequently:
$$
\sum_{k=1}^{r} m_{k} \, {\rm Re}(\lambda_{k}) = \Tr (S)\,.
$$
Now, using the fact that $\sum_{k=1}^{r} m_{k} = N$\,, we conclude
$$
\min {\rm Re} \big[ \sigma (\tilde{B}_{J}) \big] = \min \left\{ {\rm Re}
(\lambda_k) , k \in \{1,\ldots,r\} \right\} \leq
\frac{\Tr(S)}{N}\,.
$$
\end{proof}

\subsection{Optimization of $\sigma(\tilde{B}_{J})$}

Our goal now is to maximize $\min {\rm Re} \big[ \sigma (B_{J})
\big]$ over $J \in {\mathcal A}_N(\R)$\,, or equivalently, to maximize
$\min {\rm Re} \big[ \sigma (\tilde{B}_{J})\big]$ over
$\tilde{J}=S^{1/2}JS^{1/2} \in {\mathcal A}_N(\R)$\,. 
 Indeed, this is the quantity which will determine the exponential
 rate of convergence to equilibrium of the non-reversible
 dynamics~\eqref{e:nonrev_lin} as it will become clear below. 

From Proposition~\ref{prop:spectrumB}$(iii)$,  the maximum
is obviously achieved if there exists a matrix $J\in {\mathcal A}_N(\R)$ such that:
\begin{equation}
\label{e:maximize} 
\forall \; \lambda \in
\sigma(\tilde{B}_{J}), \quad {\rm Re}(\lambda) = \frac{\Tr(S)}{N}\,.
\end{equation}
In the following proposition we obtain a characterization of
the antisymmetric matrices $\tilde{J}$ (related to $J$ through $\tilde{J}=S^{1/2}JS^{1/2}$) for which~\eqref{e:maximize} is
satisfied and $\tilde{B}_{J}$ is diagonalizable (see~\eqref{e:J1}
below). This characterization requires to
introduce a companion real symmetric positive definite matrix $Q\in
\mathcal{S}_{N}^{>0}(\R)$. The case of non-diagonalizable
$\tilde{B}_{J}$ is then discussed, using an asymptotic argument.
We finally show how this characterization can be used to develop an algorithm for constructing a matrix 
$ \tilde{J} \in \mathcal{A}_{N}(\R)$ such that~\eqref{e:maximize} is satisfied. 
\begin{proposition}\label{prop:maximize} Assume that $\tilde{J}\in
  \mathcal{A}_{N}(\R)$ and that $S\in \mathcal{S}_{N}^{>0}(\R)$\,. Then the
following conditions are equivalent:
\begin{enumerate}
\item[(i)] The matrix $\tilde{B}_{J}=S+\tilde{J}$ is diagonalizable (in $\bC$) and the spectrum of $\tilde{B}_{J}$ satisfies
\begin{equation}\label{e:BJ_spectrum} \sigma (\tilde{B}_{J}) \subset
\frac{\Tr(S)}{N} + i \R\,.
\end{equation}
\item[(ii)] $\displaystyle{\tilde{B}_{J} - \frac{\Tr(S)}{N} I}$ is similar to an anti-adjoint
matrix.
\item[(iii)] There exists a hermitian positive definite matrix $Q = \overline{Q}^{T}$
such that
\begin{equation}\label{e:J1C} \tilde{J}Q - Q\tilde{J} = -QS -SQ + \frac{2
\Tr(S)}{N} Q\,.
\end{equation}
\item[(iv)] There exists a real symmetric positive definite matrix $Q = Q^{T}$
such that
\begin{equation}\label{e:J1} \tilde{J}Q - Q\tilde{J} = -QS -SQ + \frac{2
\Tr(S)}{N} Q\,.
\end{equation}
\end{enumerate}
\end{proposition}
\begin{proof} First we prove the equivalence between $(i)$ and
$(ii)$. Equation~\eqref{e:BJ_spectrum} is equivalent to the statement that
there exists a matrix $P \in GL_{n}(\bC)$ (where $GL_{n}(\bC)$ denotes
the set of complex valued invertible matrices) such that
$$
P^{-1} \left( \tilde{B}_{J} - \frac{\Tr(S)}{N} I \right) P =
\mbox{diag}(i t_{1}, \dots i t_{N})
$$
for some $t_k$ in $\R$\,,
which is equivalent to statement $(ii)$, since any anti-adjoint matrix can be diagonalized in $\bC$\,.

To prove that $(ii)$ implies $(iii)$, we write statement $(ii)$ as: there
exists a matrix $P \in GL_{n}(\bC)$ such that
\begin{equation}\label{eq:1}
\left(\overline{P^{-1} \tilde{B}_{J} P}\right)^{T} - \frac{\Tr(S)}{N} I = - P^{-1}
\tilde{B}_{J} P + \frac{\Tr(S)}{N} I\,.
\end{equation}
Since $\tilde{B}_{J}=S+\tilde{J}\in \mathcal{M}_{N}(\R)$ and $\tilde{J} \in {\mathcal A}_N(\R)$\,, we obtain
$$
P^{-1} \tilde{J} P -  \overline{P}^{T}\tilde{J} \left( \overline{P}^{-1} \right)^{T} =
-\overline{P}^{T} S \left( \overline{P}^{-1} \right)^{T} -P^{-1} S P + \frac{2
\Tr(S)}{N} I\,.
$$
We multiply this equation left and right by $P$ and $\overline{P}^{T}$
respectively, to obtain
\begin{equation}\label{eq:2}
\tilde{J} P \overline{P}^T - P \overline{P}^T \tilde{J} = - P \overline{P}^T S - S P
\overline{P}^T + \frac{2 \Tr(S)}{N} P \overline{P}^T\,.
\end{equation}
Statement $(iii)$ follows now by taking $Q = P
\overline{P}^T$\,. Conversely,  $(iii) \Rightarrow (ii)$
follows from the writing  $Q=P
\overline{P}^T$\,, with $P \in GL_{n}(\bC)$ (take $P=\sqrt{Q}$)
for any hermitian positive definite matrix
$Q$\,. Then, one obtains $(ii)$ by going back
from~\eqref{eq:2} to~\eqref{eq:1}. 
Finally, $(iii)$ implies $(iv)$ by taking the real part of~\eqref{e:J1C}
and using the fact that $\tilde{J}$ and $S$ are real matrices. The
converse $(iv) \Rightarrow (iii)$ is obvious.
 This ends the proof.
\end{proof}

\begin{remark}\label{rem:signJ}
Notice that if $\tilde{J}$ is such
that~\eqref{e:BJ_spectrum} is satisfied, so is $-\tilde{J}$ (and thus $\tilde{J}^T$). Indeed,
if $(\tilde{J},Q)$ satisfies~\eqref{e:J1}, then $(-\tilde{J},Q^{-1})$
also satisfies~\eqref{e:J1}.
\end{remark}

Let us give another equivalent formulation of~\eqref{e:J1}.
\begin{lemma} 
With the notation of Proposition~\ref{prop:maximize}, let us consider
matrices $\tilde J \in {\mathcal A}_N(\R)$, $S \in {\mathcal
  S}^{>0}_N(\R)$ and $Q \in {\mathcal S}^{>0}_N(\R)$. Let us denote $\{\lambda_{k}
\}_{k=1}^{N}$ the positive real eigenvalues of $Q$ (counted with multiplicity), and $\{\psi_{k} \}_{k=1}^{N}$ the associated eigenvectors, which form an
orthonormal basis of $\R^{N}$. Equation~\eqref{e:J1} is equivalent to the two conditions: for all $k$ in $\{1, \dots, N \}$,
\begin{equation}\label{e:cond_nec} 
\left( \psi_{k} , S \psi_{k}\right)_{\R} = \frac{\Tr(S)}{N}
\end{equation}
and, for all $j \neq k$ in $\{1, \dots, N \}$,
\begin{equation}
\label{e:cond_J} 
(\lambda_{j} - \lambda_{k}) ( \psi_{j},\tilde J \psi_{k})_{\R}= (\lambda_{k} +
\lambda_{j}) ( \psi_{j} , S \psi_{k}
)_{\R}\,.
\end{equation}
\end{lemma}
\begin{proof} 
Since  $\{\psi_{k} \}_{k=1}^{N}$ form an
orthonormal basis of $\R^{N}$, Equation~\eqref{e:J1} is equivalent to this same equation tested against $\psi_j^T$ on the left, and $\psi_k$ on the right. This yields:
$$
\lambda_k \psi_j^T \tilde{J} \psi_k - \lambda_j \psi_j^T \tilde{J} \psi_k= - \lambda_j \psi_j^T S \psi_k - \lambda_k \psi_j^T S \psi_k + \frac{2
\Tr(S)}{N} \delta_{jk} \lambda_k\,,
$$
where $\delta_{jk}$ is the Kronecker symbol. When $j=k$\,, we obtain~\eqref{e:cond_nec} by using the
antisymmetry of $\tilde J$\,, together with the fact that all
eigenvalues of $Q$ are non-zero. When $j \neq k$\,, we obtain~\eqref{e:cond_J}.
\end{proof}
Notice that when the eigenvalues of $Q$ are all with multiplicity one, $\tilde J$ is completely determined by~\eqref{e:cond_J}: for all $j \neq k$ in $\{1,\ldots,N\}$\,,
\begin{equation}
\label{e:Jsoln} 
( \psi_{j}\,,\,\tilde J \psi_{k})_{\R}= - \frac{\lambda_{k} +
\lambda_{j}}{\lambda_{k} - \lambda_{j}} ( \psi_{j} , S \psi_{k}
)_{\R}\,.
\end{equation}
Indeed, by the antisymmetry of $\tilde J$\,, the remaining entries are
zero: 
$$
( \psi_{j}\,,\,\tilde J \psi_{j})_{\R}=0\quad\text{for~all}~j \in
\{1,\ldots,N\}\,.
$$
This motivates the following definition.
\begin{definition}
\label{def:Popt}
We will denote by
  $\mathcal{P}_{opt}(S)$  the set of pairs $(\tilde{J},Q)$, where $Q$ is a real symmetric
  positive definite
  matrix with  $N$ eigenvalues of multiplicity one and associated eigenvectors
  satisfying~\eqref{e:cond_nec}, and $\tilde{J}$ is the associated
  antisymmetric matrix defined by~\eqref{e:Jsoln}.
\end{definition}
Notice that for any $(\tilde{J},Q) \in \mathcal{P}_{opt}(S)$\,,
$\tilde{J}$ is completely defined (by~\eqref{e:Jsoln}) as soon as $Q$
is chosen, so that the set $\mathcal{P}_{opt}(S)$ can be indexed by the set of
  matrices $Q \in {\mathcal S}_N^{>0}(\R)$ with $N$
  eigenvalues of multiplicity one, and with eigenvectors $\psi_k$
  satisfying~\eqref{e:cond_nec}.
 As it will become clear below, the matrix $Q$ of a 
pair $(\tilde J,Q)\in \mathcal{P}_{opt}(S)$ appears in the 
quantitative estimates of Theorem~\ref{thm:lin_alg} and
Theorem~\ref{th:CV_KB} through the constants $C_N^{(1)}$ and $C_{N}^{(2)}$\,. The construction
of the pair $(\tilde{J},Q)$ is also better understood by splitting the
two steps: (1) construction of $Q$\ and (2) when $Q$ is fixed, construction
of $\tilde{J}$\,.

\begin{remark}\label{rem:non_diag}
We would
like to stress that the set $\mathcal{P}_{opt}(S)$  does not provide {\em all} the
matrices $\tilde{J}\in \mathcal{A}_{N}(\R)$ such that
$\sigma(\tilde{B}_{J})
\subset \frac{\Tr(S)}{N} + i \R$\,. Indeed, first, we have assumed that
$\tilde{B}_{J}$ is diagonalizable and, second, in this case we have
assumed moreover that
$Q$ has $N$ eigenvalues of multiplicity one.

Actually the spectrum of
$\tilde{B}_{J}$ depends continuously on $\tilde{J}$. Hence any limit
$\tilde{J}=\lim_{n\to\infty}\tilde{J}_{n}$ in $\mathcal{A}_{N}(\R)$
with $(\tilde{J}_{n}, Q_{n})\in \mathcal{P}_{opt}(S)$ will lead to
$\sigma(\tilde{B}_{J})\subset \frac{\Tr(S)}{N} + i \R$\,. A
particular case is interesting: Fix the real orthonormal basis
$\{\psi_{j}\}_{j=1}^{N}$ and consider $Q_{\alpha}$ with the
eigenvalues $(\alpha,\ldots, \alpha^{N})$ with $\alpha>0$\,.
The unique associated antisymmetric matrix $\tilde{J}_{\alpha}$ is given by
$(\psi_{j}\,,\,\tilde{J}_{\alpha}\psi_{j})_{\R}=0$ and
$$
(\psi_{j}\,,\,\tilde{J}_{\alpha}\psi_{k})_{\R}=
-\frac{\alpha^{k}+\alpha^{j}}{\alpha^{k}-\alpha^{j}}
(\psi_{j}\,,\,S\psi_{k})_{\R}\,.
$$
Taking the limit as $\alpha\to +\infty$ or $\alpha\to 0^{+}$ leads to
$$
(\psi_{j}\,,\,\tilde{J}_{\infty}\psi_{k})_{\R}
=- \mbox{sign}(k-j)(\psi_{j}\,,\,S\psi_{k})_{\R},\quad
\tilde{J}_{0^{+}}=-\tilde{J}_{\infty}\,.
$$
Actually, for such a choice
$\tilde{J}_{opt}=\tilde{J}_{\infty}$ or
$\tilde{J}_{opt}=\tilde{J}_{0^{+}}$\,, 
the matrix $S+\tilde{J}_{opt}$
is triangular in the basis $(\psi_{j})_{1\leq j\leq N}$ and 
$\sigma(\tilde{B}_{J_{opt}})=\left\{\frac{\Tr(S)}{N}\right\}$\,. In
general (see for example Subsection~\ref{sec:2D}), the matrix
$\tilde{B}_{J_{opt}}$ may not be diagonalizable over $\bC$
and may have Jordan blocks.
\end{remark}

We end this section by providing a practical way to construct a couple
$(\tilde J, Q)$ satisfying~\eqref{e:J1} (or equivalently $(\tilde
J,Q)\in \mathcal{P}_{opt}(S)$), for a given $S \in   \mathcal{S}_{N}^{>0}(\R)$. The strategy is simple. We first build an orthonormal basis $\{\psi_{k} \}_{k=1}^{N}$ of $\R^{N}$ such
that~\eqref{e:cond_nec} is satisfied, then we {\em choose} the
eigenvalues $\{\lambda_k\}_{k=1}^{N}$ distinct and positive, and
define $\tilde{J}$ by~\eqref{e:Jsoln}. The only non-trivial task is
thus to build the orthonormal basis $\{\psi_{k} \}_{k=1}^{N}$\,.

\begin{proposition}\label{prop:basis} 
For every $S \in  \mathcal{S}_{N}^{>0}(\R)$,
 there exists an orthonormal basis $\{\psi_{k} \}_{k=1}^{N}$ of $\R^{N}$ such
that~\eqref{e:cond_nec} is satisfied.
\end{proposition}
\begin{proof}
We  proceed by
induction on $N$, using some Gram-Schmidt orthonormalization
process. The result is obvious for $N=1$\,. For a positive integer
$N$\,, 
let us assume it is true for
$N-1$ and let us consider  $S \in  \mathcal{S}_{N}^{>0}(\R)$\,. Let us
set $T = \frac{S}{\Tr(S)}$\,. 
The matrix $T$ is in $ \mathcal{S}_{N}^{>0}(\R)$ with $\Tr(T) =
1$\,. Consequently $( \psi_{i}, T \psi_{i} )_{\R} >0, \; i=1,
\dots N $ and $\sum_{i=1}^{N} ( \psi_{i}, T \psi_{i} )_{\R} =
1$ for any orthonormal basis $\{\psi_{i} \}_{i=1}^{N}$ of $\R^N$\,.
 Assume that not all $ ( \psi_{i}, T \psi_{i}  )_{\R} $ are
equal to $1/N$\,. Then there exist $i_{0}, \, i_{1} \in \{1, \ldots, N \}$
such that
$$ 
( \psi_{i_{0}}, T \psi_{i_{0}} )_{\R} < \frac{1}{N}, \quad
( \psi_{i_{1}}, T \psi_{i_{1}} )_{\R} > \frac{1}{N}\,.
$$
Set $\psi_{t} = \cos(t) \psi_{i_{0}} + \sin(t)\psi_{i_{1}}$ and
consider the function $f(t) = ( \psi_{t} , T \psi_{t}
)_{\R}$\,. This function is continuous with $f(0) < 1/N $ and $f(\pi/2)
> 1/N$\,. Consequently, there exists a $t_{*} \in (0, \pi/2)$ such that
\begin{equation}\label{e:t1} 
( \psi_{t_{*}} , T \psi_{t_{*}})_{\R} = \frac{1}{N}\,.
\end{equation}
Let now $\Pi=I-\psi_{t_*} \left(\psi_{t_*}\right)^T$ denote the orthogonal projection to
$\mbox{Span}\big(\psi_{t_*}\big)^\perp$ and define
$$
T^{1} = \frac{N}{N-1} \Pi T \Pi\,.
$$
This operator is symmetric positive definite on $\mbox{Span}\big(\psi_{t_*}\big)^\perp$ with
\begin{align*} \Tr(T^{1}) &  =  \frac{N}{N-1}
\left(\Tr(T) - (\psi_{t_*} , T \psi_{t_*} )_{\R}
\right) = 1\,.
\end{align*}
It can thus be associated with a symmetric positive definite matrix in
${\mathcal M}_{N-1}(\R)$\,. 
By the induction hypothesis there exists an
orthonormal basis $\big(\tilde\psi_{2}, \ldots ,\tilde\psi_{N} \big)$ of
$\mbox{Span}\big(\psi_{t_*}\big)^\perp$ such that
$$
( \tilde \psi_{i}, T^{1} \tilde \psi_{i} )_{\R} = \frac{1}{N-1}, \quad i
= 2, \ldots, N\,.
$$
Let us consider the orthonormal basis of $\R^N$\,:
$$
 \tilde\psi_{i} = \left\{ \begin{array}{ c } \psi_{t_*}, \quad i=1,
\\ \tilde\psi_{i}, \quad i \geq 2.
            \end{array} \right.
$$
We obtain
$$
( \tilde\psi_{1}, T \tilde\psi_{1} )_{\R} = (
\psi_{t_*}, T \psi_{t_*} )_{\R} = \frac{1}{N}
$$
and, for $i \ge 2$\,,
$$
( \tilde\psi_{i}, T \tilde\psi_{i} )_{\R} = \frac{N-1}{N}
( \tilde\psi_{i}, T^{1} \tilde\psi_{i} )_{\R} =
\frac{1}{N}\,.
$$
This ends the induction argument.
\end{proof}
\begin{remark}\label{rem:tstar}
Finding $t_{*}$ such that~\eqref{e:t1} is satisfied yields a simple algebraic problem in two dimensions. Let
$(i_{0},i_{1})$ be the two indices introduced in the proof. The
 matrix $\big(  (\psi_{i}\,,\, T\psi_{j})_{\R} \big)_{i,j\in
 \left\{i_{0},i_{1}\right\}}\in \mathcal{M}_{2}(\R)$ is
$$
\left[
\begin{matrix}
  \alpha_{0} &\beta\\
\beta & \alpha_{1}
\end{matrix}
\right]
\quad\text{with}\quad \alpha_{0}<\frac{1}{N}\,,\,
\alpha_{1}>\frac{1}{N}\,,\, \beta\in\R\,.
$$
Then, $t_{*} \in (0,\pi/2)$ is given by
$$
\tan t_{*}=\frac{-\beta+\sqrt{\beta^{2}-\left(\alpha_{1}-\frac{1}{N}\right)\left(\alpha_{0}-\frac{1}{N}\right)}}{\alpha_{1}-\frac{1}{N}}
$$
and the vector $\psi_{t_{*}}$ by
$$
\psi_{t_{*}}=\frac{1}{\sqrt{1+\tan^{2}t_{*}}}(\psi_{i_{0}}+\tan t_{*}\psi_{i_{1}})\,.
$$
\end{remark} 

%
The above proof and Remark~\ref{rem:tstar} yield a practical
algorithm, in the spirit of the Gram-Schmidt procedure,
 to build an
orthonormal basis satisfying~\eqref{e:cond_nec}, see Figure~\ref{fig:algor}. This algorithm
is  used for the numerical experiments of
Section~\ref{sec:num}. Notice that in the third step of the algorithm,
only the vector $\psi_{n+1}$ is concerned by 
the Gram-Schmidt procedure. The chosen vector $\psi_{t_{*}}$
belongs to $\R \psi_{n}\oplus \R \psi_{n+1}$ and all the normalized
vectors $(\psi_{n+2},\ldots,\psi_{N})$ are already orthogonal to this plan.
\begin{figure}[htbf]
 \fbox{
\begin{minipage}[htbf]{12cm}
\begin{center}
 {\bf Algorithm for constructing the optimal nonreversible perturbation}
\end{center}
\begin{description}
\item Start from an arbitrary orthonormal basis $(\psi_1, \ldots, \psi_N)$.
\end{description}
{\bf for} $n = 1 :N-1   $ {\bf do}
\begin{enumerate}
\item Make a permutation of $(\psi_{n}, \ldots, \psi_N)$ so that
  $$(\psi_n,S\psi_n)_{\R} = \max_{k=n, \ldots, N}(\psi_k,S\psi_k)_{\R}
  > \Tr(S) / N$$
  and   
$$
(\psi_{n+1},S\psi_{n+1})_{\R} = \min_{k=n, \ldots,
    N}(\psi_k,S\psi_k)_{\R} < \Tr(S) / N\,.
$$
\item Compute $t_*$ such that $\psi_{t_*}=\cos(t_*) \psi_n + \sin(t_*)
  \psi_{n+1}$ satisfies $(\psi_{t_*},S\psi_{t_*})_{\R} = \Tr(S)/N$
  (see Remark~\ref{rem:tstar} above).
\item Use a Gram-Schmidt procedure to change the set of vectors
  $(\psi_{t_*},\psi_{n+1}, \ldots, \psi_N)$ to an orthonormal basis 
$(\psi_{t_*},\tilde \psi_{n+1}, \ldots, \tilde \psi_N)$\,. 
\end{enumerate}
{\bf end}
\end{minipage}
}
\caption{Algorithm for constructing the optimal nonreversible perturbation}\label{fig:algor}
\end{figure}

A simple corollary of Proposition~\ref{prop:basis} is the following:
\begin{proposition}\label{pr:opt}
For every $S \in {\mathcal S}_N^{>0}(\R)$\,, it is possible to build a
matrix $\tilde{J} \in {\mathcal A}_N(\R)$ such that
$$
\frac{\Tr (S)}{N}=\min {\rm Re}[\sigma(\tilde{B}_J)] \ge \min {\rm
  Re}[\sigma(S)]
$$
where $\tilde{B}_J=S+\tilde{J}$\,. Moreover, this holds with a strict inequality as
soon as
$S$ admits two different eigenvalues.
\end{proposition}
In conclusion, the exponential rate of
convergence may be improved by using a non-reversible perturbation, if and only if
$S$ is not proportional to the identity.
We also refer to~\cite[Theorem 3.3]{Hwang_al1993} for another
characterization of the strict inequality case.

\subsection{Explicit computations in the two dimensional case}
\label{sec:2D}

In the two dimensional case ($N=2$), 
all the matrices $J$ such that $\sigma(B_J) \subset \Tr(S)/N + i \R$
can be characterized. Accordingly, explicit accurate estimate of the exponential decay
are available for the two-dimensional ordinary differential equation:
\begin{equation}
  \label{eq.ODE2}
\frac{dx_t}{dt}=-(I + J) S x_t \quad\text{ with }\quad
x_0~\text{given~in}~\R^{2}\,.
\end{equation}
After making the connection with our general construction of the
optimal matrices $J$ (see Definition~\ref{def:Popt}), we investigate, for a given
matrix $S \in  \mathcal{S}_{N}^{>0}(\R)$\,, the minimization, with
respect to $J$\,, of the prefactor in
 the exponential decay law. We would like in particular to discuss the optimization of
the constant factor in front of $\exp(-\Tr (S) t / 2)$\,.\\
Without loss of generality, we may assume that
$$S=\left[\begin{array}{cc} 1 & 0 \\ 0 & \lambda \end{array} \right] \text{ and } J=\left[\begin{array}{cc} 0 & a \\ -a & 0 \end{array} \right],$$
where $\lambda >0$ is fixed.
The eigenvalues of $B_{J}$ belong to $\Tr(S)/2=(1+\lambda)/2+i\R$ if
and only if
\begin{equation}\label{eq:cond_a}
a^2 \ge \frac{(1-\lambda)^2}{4\lambda}
\end{equation}
and then, the eigenvalues of $B_J=(I+J)S$ are
$$
\mu_{\pm}=\frac{\lambda +1 \pm i \sqrt{4\lambda a^2 -
    (1-\lambda)^2}}{2}
\,.
$$
When the inequality \eqref{eq:cond_a} is strict, the associated eigenvectors are
$$u_{\pm}=\left( \begin{array}{c} 1 \\ \alpha_{\pm} \end{array}
\right) \text{ with } \alpha_{\pm}=\frac{\mu_{\pm}-1}{a
  \lambda}=\frac{\lambda-1 \pm i \sqrt{4\lambda a^2 -
    (1-\lambda)^2}}{2a \lambda}
\,.
$$
The matrix $B_{J}$ equals
\begin{eqnarray*}
  &&B_J=P \left[\begin{array}{cc} \mu_+ & 0 \\ 0 & \mu_- \end{array}
  \right]P^{-1} \,,\\
\text{with}
&&P=\left[\begin{array}{cc} 1 & 1 \\ \alpha_+ & \alpha_- \end{array} \right] \text{ and } P^{-1}=\frac{1}{\alpha_- - \alpha_+}\left[\begin{array}{cc} \alpha_- & -1 \\ -\alpha_+ &1 \end{array} \right]\,.
\end{eqnarray*}
The case $a=\pm\frac{(1-\lambda)}{2\sqrt{\lambda}}$ gives the matrix 
$$
B_{J}=
\begin{pmatrix}
  1 &\pm\frac{\sqrt{\lambda}(1-\lambda)}{2}\\
\mp \frac{1-\lambda}{2\sqrt{\lambda}}&\lambda
\end{pmatrix}
$$
which has a Jordan block when $\lambda\neq 1$\,.
This ends the characterization of all the possible optimal $J$'s in
terms of the exponential rate.\\

Let us compare with the general construction of the pair
$(Q,\tilde{J}=S^{\frac 1 2}JS^{\frac 1 2})$, see
Definition~\ref{def:Popt}. The matrix $Q$ is diagonal in an orthonormal basis
$(\psi_{1},\psi_{2})$ which satisfies the
relation \eqref{e:cond_nec}. This yields
$$
|\psi_{1}^{1}|^{2}=|\psi_{1}^{2}|^{2}=|\psi_{2}^{1}|^{2}=|\psi_{2}^{2}|^{2}=\frac
1 2\,.
$$
Up to trivial symmetries one can fix 
$\psi_{1}=
\begin{pmatrix}
  \frac{1}{\sqrt 2}\\ \frac{1}{\sqrt 2}
\end{pmatrix}
$ 
and 
$
\psi_{2}= 
\begin{pmatrix}
  \frac{1}{\sqrt 2}\\ 
- \frac{1}{\sqrt 2}
\end{pmatrix}
$\,.
Then, from~\eqref{e:cond_J}, the eigenvalues $\lambda_{1}, \lambda_{2}$ of $Q$ must satisfy
$$
(\lambda_{2}-\lambda_{1})(-2a\sqrt{\lambda})=(\lambda_{2}+\lambda_{1})(1-\lambda)
$$
and the limiting cases $a=\pm\frac{1-\lambda}{2\sqrt{\lambda}}$ are
achieved only after taking the limit
$\frac{\lambda_{2}}{\lambda_{1}}\to +\infty$ or
$\frac{\lambda_{1}}{\lambda_{2}}\to +\infty$\,.\\
Assume now $a^2 > \frac{(1-\lambda)^2}{4\lambda}$ and consider
the two-dimensional Cauchy problem~\eqref{eq.ODE2}. Its solution equals
$$x_t=P \left[\begin{array}{cc} \exp(-\mu_+t) & 0 \\ 0 & \exp(-\mu_-
    t) \end{array} \right]P^{-1} x_0
\,,$$
which leads to 
$$
\|x_t\| \le \|P\| \|P^{-1}\| \exp\left(-\frac{1+\lambda}{2}t\right)
\|x_0\|\,,
$$
when $\|\cdot\|$ denotes either the Euclidean norm on vectors or the
associated matrix norm\,, $\|A\|=\sqrt{\max(\sigma(A^* A) )}$\,.
This yields the exponential convergence with rate
$\Tr(S)/2=(1+\lambda)/2$, as soon as $a$
satisfies~$a^{2}>\frac{(1-\lambda)^{2}}{4\lambda}$\,, while the
degenerate case $a^{2}=\frac{(1-\lambda)^{2}}{4\lambda}$ would give
an upper bound $C(1+t)e^{-\frac{1+\lambda}{2}t}$\,.
A more convenient matrix norm is the Frobenius norm given by
$\|A\|_{F}^{2}=\sum_{i,j=1}^{2}|A_{ij}|^{2}=\sum_{\alpha\in
  \sigma(A^{*}A)}\alpha$ with the equivalence in dimension~$2$\,,
$\frac{1}{\sqrt{2}}\|A\|_{F}\leq \|A\|\leq
\|A\|_{F}$\,.
By recalling
$\alpha_{+}=\overline{\alpha_{-}}$\,, we get
\begin{eqnarray*}
\|x_t\| 
&\leq&
\|P\|_{F} \|P^{-1}\|_{F} \exp\left(-\frac{1+\lambda}{2}t\right)\|x_{0}\|
\\
&\leq& 2  \frac{(1+|\alpha_+|^2)}{|\alpha_--\alpha_+|}
\exp\left(-\frac{1+\lambda}{2}t\right)\|x_{0}\|
\\
&\leq &
 2 (\lambda+1)  \frac{|a|}{\sqrt{4  \lambda a^2-(1-\lambda)^2}}
\exp\left(-\frac{1+\lambda}{2}t\right)\|x_{0}\|\,.
\end{eqnarray*}
Now, it is clear that the infimum of $\|P\|_F \|P^{-1}\|_F$
 is obtained asymptotically as $|a| \to \infty$ and equals
 $\frac{\lambda+1}{\sqrt{\lambda}}$\,. It corresponds to an antisymmetric matrix $J$ with infinite norm.

To end this section, we would like to discuss the situation when the
original dynamics (when $J=0$) has two separated time scales, namely
$\lambda$ is very large or very small. In the case $\lambda \ll 1$\,, we
observe that the optimal $\|P\|_F \|P^{-1}\|_F$ (and thus the optimal
$\|P\| \|P^{-1}\|$) scales like $\frac{1}{\sqrt{\lambda}}$\,, and that
this scaling in $\lambda$ is already achieved by taking
$a^2=\frac{(1-\lambda)^2}{2 \lambda}$ (twice the minimum value
in~\eqref{eq:cond_a}), since in this case, $\displaystyle\|P\|_F
\|P^{-1}\|_F=\sqrt{2} \frac{(\lambda+1)}{\sqrt{\lambda}}$\,. In terms of
rate of convergence to equilibrium, it means that, to get $\|x_t\|$ of
the order of $\|x_0\|/2$, say, it takes a time of order
$\ln(1/\lambda)$. This should be compared to the original dynamics
(for $a=0$), for which this time is of order $1/\lambda$\,. Of course, a
similar reasoning holds for $\lambda \gg 1$\,. Using an antisymmetric
perturbation of the original dynamics, we are able to
dramatically  accelerate convergence to equilibrium.

\section{Convergence to equilibrium for Gaussian laws and applications}\label{sec:gauss}

In this section, we use the results of the previous section in order
to understand the longtime behavior of the mean and the covariance of
$X_t$ solution to~\eqref{e:nonrev_quad}:
$$
dX_t=-(I+J)SX_t \, dt+\sqrt{2}\, dW_{t}\,.
$$ 
In particular, if $X_0$ is a Gaussian random variable (including the case where $X_{0}$ is deterministic), then $X_t$ remains a Gaussian
random variable for
all times, and understanding the longtime behavior of the mean $\E(X_t)$ and the covariance matrix  $\mbox{Var}(X_{t}) = \E(X_t \otimes X_t) - \E(X_t) \otimes \E(X_t)$ is equivalent to
understanding the longtime behavior of the density of the process $X_t$,
which is exactly Corollary~\ref{cor:CV_FP} in a very specific case.  Here and in the
following, $\otimes$ denote the tensor product: for two vectors $x$
and $y$ in $\R^N$, $x \otimes y=x y^T$ is a $N \times N$ matrix with
$(i,j)$-component $x_i y_j$\,.

\subsection{The mean}

Let us denote $x_t=\E(X_t)$\,, which is the solution to the ordinary differential equation
\begin{equation}\label{e:ode_symm}
\frac{d x_{t}}{d t} = - (I+J) S x_{t}\,,
\quad x_{0} = x\,.
\end{equation}
The longtime behavior of $x_t$ amounts to getting appropriate bounds on the
semigroup $e^{-(I + J) S t}$ or equivalently on $e^{-(S+\tilde{J})t}$\,.

When $J=0$\,, namely for the ordinary differential equation
\begin{equation*}
\frac{d x_{t}}{d t} = - S x_{t}\,,
\quad x_{0} = x\,,
\end{equation*}
we immediately deduce from the spectral representation of the positive
symmetric matrix $S$  that
\begin{equation}\label{e:conv_symm} 
\|x_{t}\| \leq e^{-\rho t}
\|x_{0} \|\,,
\end{equation} where
$$
\rho := \min \left\{\sigma (S) \right\}\,.
$$
The above bound implies that
$$
\left\|e^{-St} \right\| \leq e^{-\rho t}\,,
$$
where $\|M\|=\sup_{x \in \R^N, x\neq 0}\frac{\|Mx\|}{\|x\|}$\,. Notice that $\rho \le
\frac{\Tr (S)}{N}$\,.

We now derive a similar estimate for the semigroup generated by the perturbed matrix
$\tilde{B}_{J}=S+\tilde{J}$ (or equivalently $B_J=(I+J)S$),
when $(\tilde{J},Q)\in \mathcal{P}_{opt}$\,, and show that a better
exponential rate of convergence is obtained. As explained in the introduction, the price to pay for the improvement in the rate of convergence is the worsening of the constant (which is simply $1$ in the reversible case) in front of the exponential.
\begin{proposition}
\label{prop:conv_ode}
For
 $(\tilde{J},Q)\in
  \mathcal{P}_{opt}$ and $J=S^{-1/2}\tilde{J}S^{-1/2}$, the estimates
\begin{eqnarray}
\label{e:convode1}
\left\| e^{-(S+\tilde{J}) t}
\right\| &\leq &
\kappa(Q)^{1/2} \exp \left(-\frac{\Tr (S)}{N}t \right)\,,\\
\label{e:conv_ode_estim} 
\left\| e^{-(I+J)S t}
\right\| &\leq &\kappa(Q^{-1}S)^{1/2} \exp \left(-\frac{\Tr (S)}{N}t \right)\,,
\end{eqnarray}
hold for every $t\geq 0$\,.
\end{proposition}
\begin{proof} Consider the ordinary differential equation
\begin{equation}\label{e:ode} 
\frac{d y_{t}}{d t} = - (S+\tilde{J})
y_{t}\,, \quad y_{0}= y\,.
\end{equation} 
We introduce the scalar product $( \cdot,\cdot )_{Q^{-1}} : =
 ( \cdot , Q^{-1} \cdot )_{\R}$ on $\R^{N}$ with the
 corresponding norm $\|\cdot \|_{Q^{-1}}$\,. We
calculate:
\begin{align*} 
\frac{d}{d t}\|y_{t} \|^{2}_{Q^{-1}} 
&
=  - 2 ( y_t , Q^{-1} (S + \tilde{J})
y_{t} )_{\R} \\ 
& =  - \left( y_{t} , \left(Q^{-1}
    S + Q^{-1} \tilde{J}  + S Q^{-1} - \tilde{J}Q^{-1} \right) y_{t} \right)_{\R} \\ 
& =  -
\left( y_{t} ,\frac{2 \Tr (S)}{N}Q^{-1} y_{t}
\right)_{\R} \\ 
& =  -\frac{2 \Tr (S)}{N} \| y_{t}\|^{2}_{Q^{-1}}\,.
\end{align*}
In the above, we have used the identity
$$
Q^{-1} \tilde{J} - \tilde{J}Q^{-1} = - S Q^{-1} - Q^{-1} S + \frac{2 \Tr (S)}{N} Q^{-1}
$$
which follows from~\eqref{e:J1} after multiplication on the left and on
the right by $Q^{-1}$. From the above we conclude that
$$
\| y_{t}\|^{2}_{Q^{-1}} = e^{-2\frac{\Tr S}{N}t}\|y\|^{2}_{Q^{-1}}\,.
$$
We now use the definition of the norm $\|\cdot\|_{Q^{-1}}$ to deduce
that
\begin{align*}
  \|y_{t}\|&\leq \|Q^{1/2}\| \|Q^{-1/2}y_{t}\|\\
&=\|Q^{1/2}\| \|y_{t}\|_{Q^{-1}}\\
&\leq 
 e^{-\frac{\Tr S}{N}t}\|Q^{1/2}\| \|y\|_{Q^{-1}}
\\
&\leq
 e^{-\frac{\Tr S}{N}t}\|Q^{1/2}\|\|Q^{-1/2}\| \|y\|\\
& =e^{-\frac{\Tr S}{N}t}\kappa(Q)^{1/2}\|y\|\,.
\end{align*}
For the second estimate, we set $x_{t}=S^{-1/2}y_{t}$ and obtain
\begin{align*}
\|x_{t}\| &\leq \|S^{-1/2}Q^{1/2}\| \|Q^{-1/2}y_{t}\| \\
&\leq  e^{-\frac{\Tr(S)}{N}t} \|S^{-1/2}Q^{1/2}\| \|Q^{-1/2} y \| \\ 
&\leq e^{-\frac{\Tr(S)}{N}t}\kappa(Q^{-1}S)^{1/2} \|x\|\,.
\end{align*}
\end{proof}
Proposition~\ref{prop:conv_ode} shows that, for a well chosen matrix
$J$\,, 
the mean $x_t=\E(X_t)$
converges to zero exponentially fast with a rate
$\frac{\Tr(S)}{N}$. Equation~\eqref{eq:lin_alg} in Theorem~\ref{thm:lin_alg} is a simple
corollary of~\eqref{e:conv_ode_estim} and the inequality
$\kappa(Q^{-1}S)^{1/2} \le \kappa(Q)^{1/2} \kappa(S)^{1/2}$\,, so that
$C_N^{(1)}$ in~\eqref{eq:lin_alg} can be chosen as
\begin{equation}\label{eq:CN_ODE}
C_N^{(1)}=\kappa(Q)^{1/2}\,.
\end{equation}

\begin{remark}\label{rem:CN_ODE}
Let us make a remark concerning the constant $C_N^{(1)}$
in~\eqref{eq:lin_alg}, using the upper bound~\eqref{eq:CN_ODE}. It is
possible to have $C_N$ independent of $N$\,, while keeping the norm 
of the perturbation $\tilde J$ under control. More precisely, for a given orthonormal basis $(\psi_k)$
satisfying~\eqref{e:cond_nec}, let us consider the eigenvalues
$\lambda_k=N+k$\,. On the one hand, $C_N^{(1)}$ remains small since $\kappa(Q)
=2$\,. On the other hand, using~\eqref{e:Jsoln}, we have
\begin{align*}
\| \tilde J\|_F^2 & = 2 \sum_{j< k} \left(
  \frac{\lambda_k+\lambda_j}{\lambda_k-\lambda_j} \right)^2 (\psi_j, S
\psi_k)_\R^2\\
& \le  2  ( 4N)^2 \sum_{j<k} (\psi_j, S \psi_k)_\R^2\\
& \le 16 N^2 \| S\|_F^2\,.
\end{align*}
Thus, the norm of $\tilde J$ (compared to the one of $S$) remains
linear in $N$\,. 
\end{remark}

\subsection{The covariance}

%
%

Let us again consider $X_t$ solution to~\eqref{e:nonrev_quad}, and let
us introduce the covariance
$$
\Sigma_{t} = \E( X_{t} \otimes X_{t} ) - \E( X_{t}) \otimes \E(X_{t} )\,,
$$
which satisfies the ordinary differential
equation:
\begin{equation}\label{eq:covar}
\frac{d \Sigma_{t}}{dt}= -(I + J) S \Sigma_{t} - \Sigma_{t} S (I - J) + 2 I\,.
\end{equation}
The equilibrium variance is
$\Sigma_{\infty} = S^{-1}$\,.
\begin{proposition}
  \label{pr.eqvariance} 
For
 $(\tilde{J},Q)\in
  \mathcal{P}_{opt}$ and $J=S^{-1/2}\tilde{J}S^{-1/2}$\,, the estimate
\begin{equation}\label{e:estim_var} \left\|\Sigma_{t} - S^{-1}
\right\| \leq \kappa(Q^{-1}S) \exp \left(-2\frac{\Tr
(S)}{N} t\right)\left\|\Sigma_{0}-S^{-1}\right\|
\end{equation}
holds for all $t\geq 0$\,, when the matricial norm is induced by
the Euclidean norm on $\R^{N}$.
\end{proposition}
\begin{proof}
The solution to \eqref{eq:covar} (see e.g.~\cite{Lorenzi,snyders-zakai-70}), $\Sigma_t$ is
\begin{equation}\label{e:var} \Sigma_{t} = S^{-1} + e^{ -t B_{J} }
(\Sigma_{0}-S^{-1}) e^{-tB_{J}^{T}}.
\end{equation} 
The result then follows from the estimate on $\|e^{-tB_{J}}\|$  in
Proposition~\ref{prop:conv_ode} above
and
$\|e^{-tB_{J}^{T}}\|=\|(e^{-tB_{J}})^{T}\|=\|e^{-tB_{J}}\|$\,.
\end{proof}

\subsection{Gaussian densities}

As a corollary of Proposition~\ref{prop:conv_ode} and
Proposition~\ref{pr.eqvariance}, we get the following convergence to
the gaussian density  (see~\eqref{eq:psi_infty_lin})
\begin{equation*}
\psi_\infty(x)=\frac{\det(\Sigma_\infty)^{-1/2}}{(2\pi)^{N/2}}
\exp\left(-\frac{x^T \Sigma_\infty^{-1}
    x}{2}\right)\,,\quad\text{with}\quad
\Sigma_{\infty}^{-1}=S\,.
\end{equation*}
\begin{proposition}\label{pr.gaussian}
Assume that $X_t$ solves \eqref{e:nonrev_quad} while $X_0$
is a Gaussian random variable, so that $X_t$ is a Gaussian random
variable for all time $t \ge 0$\,, with the density $\psi_{t}^{J}$\,.
Assume moreover that $J=S^{-1/2}
\tilde J S^{-1/2}$, and that $(\tilde J, Q)$ are chosen in ${\mathcal
  P}_{opt}$\,.
Then,
the inequality
\begin{align*}
\|\psi_t^J - \psi_\infty\|_{L^2(\psi_\infty^{-1})}^{2} &\le
N2^{N}e^{-2\frac{\Tr(S)}{N}(t-t_{0})}
\\
&\quad \times \left[1+\|x_{0}\|^{2}
\exp\left(2e^{-2\frac{\Tr(S)}{N}(t-t_{0})}\|x_{0}\|^{2}\right) \right]\,,
\end{align*}
holds for all  times $t$ larger than
\begin{equation}
  \label{eq.t0}
t_{0}=\frac{N}{2\Tr
  S}\ln\left[4(1+\|S\|)\kappa(Q^{-1}S)(1+\|S\Sigma_{0}\|)\right]\,.
\end{equation}

\end{proposition}
This result is related to the result stated in
Corollary~\ref{cor:CV_FP}, that will be proven
 in Section~\ref{sec:diffusion}. Corollary~\ref{cor:CV_FP}
provides a better and uniform in time quantitative information (which
has also a better behavior with respect to the dimension $N$ according to
\eqref{eq:CN_PDE} and Remark~\ref{rem:CN_PDE}). 
On the contrary, it requires more regularity than Proposition~\ref{pr.gaussian}
which does not assume  $\psi^J_0 \in
L^2(\R^N,\psi_\infty^{-1} \, dx)$\,. 
Of course, with initial data outside $L^2(\R^N,\psi_\infty^{-1} \, dx)$\,,   the convergence
estimate makes sense only for sufficiently large times (hence the
introduction of the positive time $t_0$ in Proposition~\ref{pr.gaussian}).
\begin{proof}
The Gaussian random vector $X_t$ has the
 mean $x_t$\,, which solves \eqref{e:ode_symm}, and the covariance $\Sigma_t$\,,
solution to~\eqref{eq:covar}, so that
$$
\psi^J_t(x)=\frac{\det(\Sigma_t)^{-1/2}}{(2\pi)^{N/2}} \exp\left( - \frac{(x-x_t)^T
    \Sigma_t^{-1} (x-x_t)}{2} \right)\,.
$$
When $t\geq t_{0}$\,, Proposition~\ref{pr.eqvariance} gives
$\|\Sigma_t-\Sigma_\infty\| \le \frac 1 4 \| \Sigma_\infty\|$ and thus, 
$
\|\Sigma_{\infty}^{-\frac 1 2}\Sigma_{t}\Sigma_{\infty}^{-\frac 1
  2}-I\|\leq \frac 1 4$\,, which yields $\frac{3}{4}\Sigma_{\infty}\leq
\Sigma_{t}\leq \frac{5}{4}\Sigma_{\infty}$ and
$\frac{4}{5}\Sigma_{\infty}^{-1}\leq
\Sigma_{t}^{-1}\leq \frac{4}{3}\Sigma_{\infty}^{-1}$\,.
In particular $\frac{\Sigma_{t}^{-1}}{2}\leq \Sigma_{\infty}^{-1}$
allows to compute
\begin{align*}
1&+ \|\psi_t^J - \psi_\infty\|_{L^2(\psi_\infty^{-1})}^2
=1+ \int_{\R^N} \left(\psi_t^J - \psi_\infty\right)^2 \psi_\infty^{-1}
=\int_{\R^N} \frac{(\psi_t^J)^2}{\psi_\infty} \\
&=(2\pi)^{-N/2} \frac{\det(\Sigma_t)^{-1}}{\det(\Sigma_\infty)^{-1/2}}
\int_{\R^N} \exp\left( - (x-x_t)^T
    \Sigma_t^{-1} (x-x_t)  + \frac{x^T \Sigma_\infty^{-1} x}{2} \right).
\end{align*}
We then use the relation, for $A$ and $B$ in $\mathcal{S}_{N}^{>0}(\R)$,
\begin{align*}
  (x-x_{t})^{T}A&(x-x_{t}) -x^{T}Bx\\ 
&=
 (x-(I-A^{-1}B)^{-1}x_{t})^{T}(A-B)(x-(I-A^{-1}B)^{-1}x_{t})
\\
&\quad +x_{t}^{T}\left[A-A(A-B)^{-1}A\right]x_{t}\,,
\end{align*}
with $A=\Sigma_{t}^{-1}$ and $B=\frac{\Sigma_{\infty}^{-1}}{2}$ in
order to get
\begin{align}
1&+ \|\psi_t^J - \psi_\infty\|_{L^2(\psi_\infty^{-1})}^2  \nonumber \\
&=(2\pi)^{-N/2} \frac{\det(\Sigma_t)^{-1}}{\det(\Sigma_\infty)^{-1/2}}
\times \pi^{\frac N 2} \det\left(\Sigma_t^{-1} -
\frac{\Sigma_\infty^{-1}}{2}\right)^{-1/2}  
\nonumber \\
& \quad \times \exp\left( x_t^T
\left[\Sigma_{t}^{-1}(\Sigma_{t}^{-1}-\frac{\Sigma_{\infty}^{-1}}{2})^{-1}\Sigma_{t}^{-1}-\Sigma_{t}^{-1}
\right] x_{t}\right)
\nonumber 
\\
&=\frac{1}{\det(\Sigma_{\infty}^{-1}\Sigma_{t})^{\frac 1
    2}\det(2I-\Sigma_{\infty}^{-1}\Sigma_{t})^{\frac 1 2}}\nonumber
\\
&\quad\times \exp\left( x_t^T
\left[2(2I-\Sigma_{\infty}^{-1}\Sigma_{t})^{-1}-I
\right]\Sigma_{t}^{-1} x_{t}\right)\,.
\nonumber
\end{align}
After setting $R_{t}=I-\Sigma_{\infty}^{-1}\Sigma_{t}$\,, we deduce
\begin{align}
\|\psi_t^J& - \psi_\infty\|_{L^2(\psi_\infty^{-1})}^2 \nonumber \\
\nonumber
&= \frac{1}{\det(I-R_{t}^{2})^{\frac 1 2}}-1
\nonumber
\\
&\quad +
 \frac{1}{\det(I-R_{t}^{2})^{\frac 1
     2}}\times\left[\exp\left(x_{t}^{T}[2(I+R_{t})^{-1}-I])\Sigma_{t}^{-1}x_{t}
\right)-1\right]\,. \label{eq:expo}
\end{align}
Let us start with the determinant $\det(I-R_{t}^{2})$\,. the condition $t\geq t_{0}$ and
Proposition~\ref{pr.eqvariance} give
$$
\|R_{t}\|=\|I-\Sigma_{\infty}^{-1}\Sigma_{t}\|
\leq \frac{e^{-2\frac{\Tr(S)}{N}(t-t_{0})}}{4} 
\quad\text{and}\quad \|R_{t}^{2}\|\leq
\frac{e^{-4\frac{\Tr(S)}{N}(t-t_{0})}}{16}\,.
$$
With $\|R_{t}^{2}\|\leq \frac{1}{16}$\,, we know
$$
|\ln \det(I-R_{t}^{2})| \leq |\Tr(\ln (I-R_{t}^{2}))| \leq N
\|\ln(I-R_{t}^{2})\|
\leq -N\ln (1-\|R_{t}^{2}\|)\,.
$$
We deduce
$$
\frac{1}{\det(I-R_{t}^{2})^{\frac 1 2}}\leq
(1-\|R_{t}^{2}\|)^{-\frac{N}{2}}
\leq
\left(\frac{16}{15}\right)^{\frac N 2}\leq 2^{\frac N 2}\,.
$$
Concerning the exponential term in~\eqref{eq:expo}, Proposition~\ref{prop:conv_ode} implies that the absolute value 
$\left|x_{t}^{T}\left[2(I+R_{t})^{-1}-I\right]\Sigma_{t}^{-1}x_{t}\right|$
is smaller than
$$
(1+2\|(I+R)^{-1}\|)\|\Sigma_{t}^{-1}\| \times
\kappa(Q^{-1}S)\exp\left(-2\frac{\Tr(S)}{N}t\right)\|x_{0}\|^{2}\,.
$$
The inequality $\|(1+R)^{-1}\|\leq (1-\|R\|)^{-1}\leq \frac{4}{3}$ and
the condition $t\geq t_{0}$ imply $\|\Sigma_{t}^{-1}\|\leq
\frac{4}{3}\|\Sigma_{\infty}^{-1}\|=\frac{4}{3}\|S\|$ and
$$
\left|x_{t}^{T}\left[2(I+R_{t})^{-1}-I\right]\Sigma_{t}^{-1}x_{t}\right|
\leq \frac{44}{9\times 4}e^{-2\frac{\Tr(S)}{N}(t-t_{0})}\|x_{0}\|^{2}\,.
$$
We have proved
\begin{align}
\|\psi_{t}^{I}-\psi_{\infty}\|_{L^{2}(\psi_{\infty}^{-1})}^{2}
&\leq
\left[ \left(1-\frac{e^{-4\frac{\Tr(S)}{N}(t-t_{0})}}{16} \right)^{-\frac N
    2}-1\right]
\nonumber
\\
&\quad
+
2^{\frac N 2}
\left[\exp
  \left(\frac{11}{9}e^{-2\frac{\Tr(S)}{N}(t-t_{0})}\|x_{0}\|^{2} \right)-1\right]\,.
\nonumber
\end{align}
By using $(1-x)^{-N/2}\le 1 + N 2^{N/2} x$ when $x \in (0,1/2)$ for the
first term, and $e^{y}-1\leq ye^{y}$ when $y\geq 0$ for the second term we
finally obtain
\begin{align}
\nonumber
\|\psi_{t}^{I}-\psi_{\infty}\|_{L^{2}(\psi_{\infty}^{-1})}^{2}
&\leq
N2^{\frac N 2}e^{-4\frac{\Tr(S)}{N}(t-t_{0})}
\\
&\quad
+ \frac{11}{9}2^{\frac{N}{2}}e^{-2\frac{\Tr(S)}{N}(t-t_{0})}\|x_{0}\|^{2}
\exp \left(\frac{11}{9} e^{-2\frac{\Tr(S)}{N}(t-t_{0})}\|x_{0}\|^{2} \right)\,,
\nonumber
\end{align}
which yields the result.
\end{proof}


\subsection{General initial densities}
\label{se.geninidens}

As a corollary of Proposition~\ref{pr.gaussian}, a convergence result for
a general initial probability law can be proven by using an argument
based on the conditioning by the initial data.
\begin{proposition}\label{pr.gauss}
Let $\psi^J_t$ satisfy the Fokker-Planck
equation~\eqref{eq:FP_lin}, with an initial probability law with
density $\psi_{0}^{J}$ and such that
$\int_{\R^{N}}e^{\alpha\|x\|^{2}}\psi_{0}^{J}~dx< +\infty$ for some
positive $\alpha$\,.
  Assume moreover that $J=S^{-1/2}
\tilde J S^{-1/2}$\,, that $(\tilde J, Q)$ are chosen in ${\mathcal
  P}_{opt}$ and that $t_{0}$ is given by \eqref{eq.t0}.
Then the inequality
\begin{equation}\label{eq:pr.gauss}
\|\psi_t^J - \psi_\infty\|_{L^2(\psi_\infty^{-1})}^{2}  \le N 2^{N+1}
 e^{-2\frac{\Tr
     (S)}{N}(t-t_{\alpha})}\int_{\R^{N}}e^{\alpha\|x\|^{2}}~\psi_{0}^{J}(x)dx\,,
\end{equation}
holds for all $t\geq t_{\alpha}=t_{0}+\frac{N}{2\Tr(S)}|\ln(\frac
\alpha 4)|$\,.
\end{proposition}
\begin{proof}
In all the proof, $J=S^{-1/2}
\tilde J S^{-1/2}$ is fixed, with $(\tilde J, Q)$ chosen in ${\mathcal
  P}_{opt}$\,. For $x \in \R^N$ and $t > 0$, let us denote $\phi^x_t$
the density of the
Gaussian process $X_t^x$ solution to:
$$dX_t^x=-(I + J) S X_t^x \, dt + \sqrt{2} dW_t \text{ with }
X_0^x=x\,.$$
Proposition~\ref{pr.gaussian} can be applied with
$\psi_{t}^{J}=\phi_{t}^{x}$ and $\Sigma_{0}=0$\,, so that the time
$t_{0}=\frac{N}{2\Tr
  S}\ln\left[4(1+\|S\|)\kappa(Q^{-1}S)\right]$
is
fixed.
With the decomposition
$$\psi^J_t(y) = \int_{\R^N} \phi^x_t(y) \psi^J_0(x) \, dx,$$
coming from $\phi_{0}^{x}=\delta_{x}$\,, we can write:
\begin{align*}
\| \psi^J_t - \psi_\infty \|_{L^2(\psi_\infty^{-1})}^2
&= \int_{\R^N} \frac{(\psi^J_t)^2(y)}{\psi_\infty(y)} \, dy - 1\\
&= \int_{\R^N} \frac{1}{\psi_\infty(y)} \left( \int_{\R^N} \phi^x_t(y)
  \psi^J_0(x) \, dx \right)^2 \, dy - 1\\
&\le \int_{\R^N} \int_{\R^N} \frac{(\phi^x_t(y))^2}{\psi_\infty(y)} dy \,
  \psi^J_0(x) \, dx - 1\\
&= \int_{\R^N} \int_{\R^N} \left(
  \frac{(\phi^x_t(y))^2}{\psi_\infty(y)} dy -1 \right)\,
  \psi^J_0(x) \, dx\,.
\end{align*}
With Proposition~\ref{pr.gaussian}, we deduce 
\begin{align}
  \nonumber
\| \psi^J_t - \psi_\infty \|_{L^2(\psi_\infty^{-1})}^2
\leq &  N2^{N}e^{-2\frac{\Tr(S)}{N}(t-t_{0})}\\
&\;\times
\left[1+\int_{\R^{N}}\|x\|^{2}\exp\left(2e^{-2\frac{\Tr(S)}{N}(t-t_{0})}\|x\|^{2}\right)\psi_{0}^{J}(x)~dx\right]
\nonumber
\\
\leq& 
N2^{N}e^{-2\frac{\Tr(S)}{N}(t-t_{0})}\left[
1+\frac{1}{\alpha}\int_{\R^{N}}e^{\alpha\|x\|^{2}}\psi_{0}^{J}(x)~dx
\right] \nonumber \\
\leq& 
N2^{N}e^{-2\frac{\Tr(S)}{N}(t-t_{0})}\left(
1+\frac{1}{\alpha} \right)
\int_{\R^{N}}e^{\alpha\|x\|^{2}}\psi_{0}^{J}(x)~dx\, ,\nonumber 
\nonumber
\end{align}
for $t\geq t_{\alpha}=t_{0}+\frac{N}{2\Tr(S)}\left|\ln(\frac \alpha
  4)\right|$. To get the second line, we used (for $t \ge t_\alpha$) $ e^{-2\frac{\Tr(S)}{N}(t-t_{0})}
\le \frac \alpha 2$ and (for $u>0$) $ue^{\frac{\alpha}{2}u}\leq 
\frac{2}{ e\alpha} e^{\alpha u} \leq \frac{1}{ \alpha} e^{\alpha u} $\,.
Writing
$e^{-2\frac{\Tr(S)}{N}(t-t_{0})}=e^{-2\frac{\Tr(S)}{N}(t-t_{\alpha})}e^{-|\ln
  \frac \alpha 4|}$ and discussing the two cases $\alpha \geq 4$ and
$\alpha\leq 4$ yield the result.
\end{proof}
The aim of the analysis using Wick calculus in
Section~\ref{sec:diffusion} 
is to obtain more accurate and uniform in time
estimates.

%
%

\section{Convergence to equilibrium for initial data in
  $L^2(\R^N,\psi_\infty dx)$}
\label{sec:diffusion}
We shall study the spectral properties, and the norm estimates
of the corresponding semigroup, for the generator
$\tilde{\mathcal{L}}_{J}$ defined by~\eqref{eq:deftLJ} (with $y$
replaced by $x$ as a dummy variable):
\begin{equation*}
\tilde{\mathcal{L}}_{J}=
\nabla^{T}S\nabla-\frac{1}{4}x^{T}Sx+\frac{1}{2}\Tr(S)
+\frac{1}{2}(x^{T}\tilde{J}\nabla -\nabla^{T}\tilde{J}x)\,.
\end{equation*}
The operator $\tilde{\mathcal{L}}_{J}$  acts in $L^{2}(\R^{N}, dx;\bC)$ and is unitarily
equivalent  (when $\tilde{J}=S^{1/2}JS^{1/2}$, and after a change of variables, see Section~\ref{sec:resc}) to
$$
\mathcal{L}_{J}=-(B_{J}x)^{T}\nabla +\Delta\quad\text{with}\quad
B_{J}=(I+J)S\,,
\quad J\in \mathcal{A}_{N}\,,
$$
acting on $L^{2}(\R^{N}, \psi_\infty dx;\bC)$\,.
Since for $J\neq 0$\,, the operator $\tilde{\mathcal{L}}_{J}$ (or
$\mathcal{L}_{J}$) is not self-adjoint, it is known
(see~\cite{Trefethen_Embree,DSZ,GGN,HeNi04,HeNi05,HeSj}) that the information about the
spectrum is a first step in estimating the exponential decay of the
semigroup, but that it has to be completed by estimates on the norm of
the resolvent. This will be carried out by using a
weighted $L^{2}$-norm associated with the constructions of the
matrices $Q$ and $J$ introduced in Section~\ref{sec:linalg}.

\subsection{Additional notation and basic properties of the semigroup
  $e^{t \tilde{\mathcal L}_J}$}\label{sec:propsemi}

Let us introduce some additional notation.
\begin{itemize}
\item We choose the right-linear and
left-antilinear convention for $L^{2}$-scalar products (or
$\mathcal{S}-\mathcal{S}'$-duality products):
$$\langle
f,g\rangle_{L^{2}}=\int_{\R^{N}}\overline{f(x)}g(x)~dx\,.
$$
\item For a multi-index $n=(n_{1},\ldots,n_{N})\in \mathbb{N}^{N}$\,, we
  will denote
  $n!=\prod_{j=1}^{N}n_{j}!$, $|n|=\sum_{j=1}^{N}n_{j}$ and when
  $X_{1},\ldots,X_{N}$ belong to a commutative algebra
  $X^{n}=\prod_{j=1}^{N}X_{j}^{n_{j}}$\,.
\item The space of rapidly decaying  complex valued $\mathcal{C}^{\infty}$ functions is
  \begin{multline*}
\mathcal{S}(\R^{N})=\Big\{f\in \mathcal{C}^{\infty}(\R^{N})\,,\,
  \forall\alpha,\beta\in \mathbb{N}^{N}\,, \exists C_{\alpha\beta}\in
  \R_{+}\,,\,
 \\
 \sup_{x\in\R^{N}}|x^{\alpha}\partial_{x}^{\beta}f(x)|\leq C_{\alpha\beta}\Big\}
\end{multline*}
and its dual is denoted $\mathcal{S}'(\R^{N})$\,.
\item The Weyl-quantization $q^{W}(x,D_{x})$ of a symbol
$q(x,\xi)\in \mathcal{S}'(\R^{2N}_{x,\xi})$ is an operator defined by
its Schwartz-kernel
$$
\left[q^{W}(x,D_{x}) \right](x,y)=
\int_{\R^{N}}e^{i(x-y).\xi}q\left(\frac{x+y}{2},\xi\right)~\frac{d\xi}{(2\pi)^{N}}\,.
$$
For example, for $q(x,\xi)=f(x)$, $q^W(x,D_x)$ is the
multiplication by $f(x)$,  for $q(x,\xi)=f(\xi)$, $q^W(x,D_x)$ is the
convolution operator $f(-i \nabla)$\,,
and for $q(x,\xi)=x^T \xi$, $q^W(x,D_x)$ is $\frac{1}{2i} (x^T \nabla +
\nabla^T x)$\,. 
\item 
The Wick-quantization of a
polynomial symbols of the variables $(z,\overline{z})$, where $z\in\bC^{N}$
is an operator defined by replacing $z_{j}$ with the so-called annihilation operator 
$a_{j}=\partial_{x_{j}}+\frac{x_{j}}{2}$ and $\overline{z}_{j}$ with
the so-called
creation operator $a_{j}^{*}=-\partial_{x_{j}}+\frac{x_{j}}{2}$.
Wick's rule implies that for  monomials involving both $z$ and $\overline{z}$, the
annihilation operators are gathered on the right-hand side and the
creation operators on the left-hand side: For given multi-indices
$\alpha,\beta\in \mathbb{N}^{N}$\,, the monomial
$\overline{z}^\alpha z^\beta$ becomes $(a^*)^\alpha
a^\beta$\,.  The properties of the
Wick calculus that we need here are reviewed in
Appendix~\ref{sec:wickcalculus}. We shall also use the vectorial
 notation
$$
a=
\begin{pmatrix}
  a_{1}\\
\vdots\\
a_{N}
\end{pmatrix}\quad,\quad
a^{*}=
\begin{pmatrix}
  a^{*}_{1}\\
\vdots\\
a^{*}_{N}
\end{pmatrix}
$$
with their transpose  $a^{T}$ and $a^{*,T}$\,.
\item The  orthogonal projection from $L^{2}(\R^{N},dx;\bC)$ onto $\bC
  e^{-\frac{|x|^{2}}{4}}$ will be denoted by $\Pi_{0}$\,.
\end{itemize}
Let us now recall a few basic properties of the semigroup
$e^{t\tilde{\mathcal{L}}_{J}}$. The Weyl symbol of
$$
-\tilde{\mathcal{L}}_{J}+\frac{\Tr (S)}{2}=
-\nabla^{T}S\nabla+\frac{1}{4}x^{T}Sx
-
\frac{1}{2}(x^{T}\tilde{J}\nabla -\nabla^{T}\tilde{J}x)
$$
is (using the fact that $\tilde{J}$ is antisymmetric)
\begin{equation}
  \label{eq:WsymbtLJ}
q_{J}(x,\xi)=\xi^{T}S\xi+\frac{x^{T}Sx}{4}
-\frac{i}{2}(x^{T}\tilde{J}\xi - \xi^{T}\tilde{J}x)
=\xi^{T}S\xi+\frac{x^{T}Sx}{4}
-ix^{T}\tilde{J}\xi\,,
\end{equation}
which is a complex quadratic form on $\R^{2N}_{x,\xi}$\,.
Besides, 
the operator $-\tilde{\mathcal{L}}_{J}$ is
the Wick quantization of a quadratic polynomial since
\begin{eqnarray}
\nonumber
-\tilde{\mathcal{L}}_{J}&=&
-\nabla^{T}S\nabla+\frac{1}{4}x^{T}Sx-\frac{1}{2}\Tr(S)-\frac{1}{2}(x^{T}\tilde{J}\nabla
-\nabla^{T}\tilde{J}x)\\
\label{eq:WicktLJ}
&=&a^{*,T}(S-\tilde{J})a\,.
\end{eqnarray}
\begin{proposition}
\label{prop:semigtLJ}
The differential operator $-\tilde{\mathcal{L}}_{J}$ is continuous
from $\mathcal{S}(\R^{N})$ into itself and from $\mathcal{S}'(\R^{N})$
into itself. Its formal adjoint is
$-\tilde{\mathcal{L}}_{-J}$\,.
With the domain $D(-\tilde{\mathcal{L}}_{J})=\left\{u\in
  L^{2}(\R^{N},dx;\bC), -\tilde{\mathcal{L}}_{J}u\in
  L^{2}(\R^{N},dx;\bC)\right\}$\,, the operator $-\tilde{\mathcal{L}}_{J}$ is a
 maximal accretive and sectorial operator in
$L^{2}(\R^{N},dx;\bC)$\,. Its resolvent is compact and its kernel
equals
$\bC e^{-\frac{|x|^{2}}{4}}$\,.
The associated semigroup $(e^{t\tilde{\mathcal{L}}_{J}})_{t\geq 0}$
has the following properties:
\begin{itemize}
\item For any $u\in
  \mathcal{S}(\R^{N})$ (resp. any $u\in \mathcal{S}'(\R^{N})$),
the map  $[0,+\infty)\ni t\mapsto e^{t\tilde{\mathcal{L}}_{J}}u$ is a
$\mathcal{S}(\R^{N})$-valued (resp. $\mathcal{S}'(\R^{N})$-valued) 
$\mathcal{C}^{\infty}$ function.
\item For any $t>0$, the operator $e^{t\tilde{\mathcal{L}}_{J}}$ sends
  continuously $\mathcal{S}'(\R^{N})$ into $\mathcal{S}(\R^{N})$\,.
\item In the orthogonal decomposition
  $L^{2}(\R^{N},dx;\bC)=\bigoplus_{k\in\mathbb{N}}^{\perp}\mathcal{D}_{k}$
  into the finite dimensional vector
  spaces spanned by Hermite functions with degree $k$:
  $$\mathcal{D}_{k}=
\mbox{Span}\left\{(a^{*})^{n}e^{-\frac{|x|^{2}}{4}}\,,~n\in \mathbb{N}^{N},
|n|=k\right\},$$ the semigroup has a block diagonal decomposition
$$
e^{t\tilde{\mathcal{L}}_{J}}=\bigoplus_{k\in\mathbb{N}}^{\perp}e^{t\tilde{\mathcal{L}}_{J}}\big|_{\mathcal{D}_{k}}\,.
$$
\end{itemize}
\end{proposition}
\begin{proof}
As a differential operator with a polynomial Weyl symbol,
$-\tilde{\mathcal{L}}_{J}$ is continuous from $\mathcal{S}(\R^{N})$
(resp. $\mathcal{S}'(\R^{N})$) into itself.  Its formal adjoints has
the Weyl symbols $\overline{q_{J}(x,\xi)}=q_{-J}(x,\xi)$ and equals 
$-\tilde{\mathcal{L}}_{-J}$.
For $k\in\mathbb{N}$, set 
\begin{multline*}
\mathcal{H}^{k}=\Big\{u\in
  L^{2}(\R^{N},dx;\bC), x^{\alpha}D_{x}^{\beta}u\in 
  L^{2}(\R^{N},dx;\bC), \\
\text{for~all}~
\alpha,\beta \in \mathbb{N}^N~s.t.~|\alpha|+|\beta|\leq k\Big\}
\end{multline*}
 and
let $\mathcal{H}^{-k}$ be its dual space. They satisfy
$$
\mathop{\cap}_{k\in\mathbb{Z}}\mathcal{H}^{k}=\mathcal{S}(\R^{N})
\quad\text{and}\quad
\mathop{\cup}_{k\in\mathbb{Z}}\mathcal{H}^{k}=\mathcal{S}'(\R^{N})\,.
$$
Since $S$ is a real symmetric positive definite matrix, the inequality
$$
|q_{J}(x,\xi)|\geq \xi^{T}S\xi + \frac{x^{T}Sx}{4}  \geq C_{S}(|\xi|^{2}+|x|^{2})
$$
implies that  the operator
$-\tilde{\mathcal{L}}_{J}$ is globally elliptic (see
\cite{Hel,Sjo,Pra1,Pra2}). Therefore, it is a bijection from
$\mathcal{H}^{k}$ onto $\mathcal{H}^{k-2}$ for any $k\in
\mathbb{Z}$. This provides the compactness of the resolvent and the
maximality property. The sectoriality  (see~\cite[Chapter~VIII]{ReSi}) comes
from
\begin{eqnarray*}
  &&\langle u\,,\,
-\tilde{\mathcal{L}}_{J}u\rangle_{L^2}=
\langle u\,,\,
a^{*,T}Sa u\rangle_{L^2} -\langle u\,,\, a^{*,T}\tilde{J}a u\rangle_{L^2}\,,\\
\text{with}
&&
\left|\langle u\,,\, a^{*,T}\tilde{J}a u\rangle_{L^2}\right|\leq
\frac{\|\tilde{J}\|}{\min \sigma(S)}\langle u\,,\, a^{*,T}Sa u\rangle_{L^2}\,.
\end{eqnarray*}
This yields (using the fact that ${\rm Re}(\langle u\,,\,
-\tilde{\mathcal{L}}_{J}u\rangle_{L^2})=
\langle u\,,\,
a^{*,T}Sa u\rangle_{L^2}$ and ${\rm Im}(\langle u\,,\,
-\tilde{\mathcal{L}}_{J}u\rangle_{L^2})=-\langle u\,,\, a^{*,T}\tilde{J}a u\rangle_{L^2}$)
\begin{equation}\label{eq:sectorial}
\forall u\in \mathcal{S}(\R^{N}),\quad \left|\mathrm{arg~}\langle u\,,\,
-\tilde{\mathcal{L}}_{J}u\rangle_{L^2}\right| \leq \theta\,,
\end{equation}
with $0\leq \tan(\theta)\leq \frac{\|\tilde{J}\|}{\min
  \sigma(S)}<+\infty$. Here and in the following, ${\rm arg}(z)$ denotes the argument
of a complex number $z$\,.\\
Then the usual contour integration technique for
sectorial operators (see for example \cite[Theorem~X.52]{ReSi} and its two
corollaries) implies that
$(-\tilde{\mathcal{L}}_{J})^{k}e^{t\tilde{\mathcal{L}}_{J}}$ is
bounded for any $k\in \mathbb{N}$ and any $t>0$\,. Combined with the
global ellipticity of $-\tilde{\mathcal{L}}_{J}$\,, this provides all
our regularity results.

The orthogonal decomposition 
$L^{2}(\R^{N},dx;\bC)=\bigoplus_{k\in\mathbb{N}}^{\perp}\mathcal{D}_{k}$
is actually the spectral decomposition for the harmonic oscillator
Hamiltonian $a^{*T}a$. From the Wick calculus (use either
$[a_{i},a_{j}^{*}]=\delta_{i j}$ or the general
formula recalled in Proposition~\ref{prop:Wickcalcul}-3 in
the appendix), we deduce 
$$
\left[a^{*,T}a\,,\, a^{*,T}(S-\tilde{J})a\right]=a^{*,T}\left[I\,,\, (S-\tilde{J})\right]a=0.
$$
This implies that the spectral subspaces $\mathcal{D}_{k}$,
$k\in\mathbb{N}$, are indeed invariant by the semigroup $e^{t\tilde{\mathcal{L}}_{J}}$.
\end{proof}

Note that with the last property, the question of estimating the
convergence to equilibrium stated in Theorem~\ref{th:CV_KB} is
equivalent to estimating the decay of the semigroup
$$
e^{t\tilde{\mathcal{L}}_{J}}(I-\Pi_{0})\quad\text{or}\quad e^{t\tilde{\mathcal{L}}_{J}}\big|_{\mathcal{D}_{0}^{\perp}}
$$
where $\Pi_{0}$ is the orthogonal projection onto $\bC
e^{-\frac{|x|^{2}}{4}}=\mathcal{D}_{0}$ (see also~\eqref{eq:CV_KB'}).

\subsection{Spectrum of $\tilde{\mathcal{L}}_{J}$}
\label{sec:spectLJ}
The result of this section is a direct application of the general
results of \cite{HiPr,Pra1,Pra2} developed after~\cite{Sjo,Hor}. See
also~\cite{ottobre_pavliotis_pravda-starov,ottobre_pavliotis_pravda-starov_preprint} where these general results are used in order to compute
the spectrum of the generator of a linear SDE with, possibly
degenerate diffusion matrix. This result was first obtained in~\cite{Metafune_al2002}
using different techniques.
\begin{proposition}
\label{prop:spectLJ}
The spectrum of the operator
$-\tilde{\mathcal L}_{J}$ equals
$$
\sigma(-\tilde{\mathcal L}_{J})=\left\{\sum_{\lambda\in
    \sigma(\tilde B_{J})}k_{\lambda}\lambda\,, k_{\lambda}\in \mathbb{N}\right\}\,,
$$
and its kernel is $\bC e^{-\frac{|x|^{2}}{4}}$\,.
\end{proposition}
\begin{proof}
The spectrum of the operator $q_{J}^{W}(x,D_{x})=-\tilde{\mathcal L}_{J}+\frac{\Tr(S)}{2}$
associated with the elliptic quadratic Weyl symbol $q_{J}(x,\xi)$
defined by~\eqref{eq:WsymbtLJ} equals, according to
\cite{HiPr}-Th~1.2.2,
$$
\sigma(q_{J}^{W}(x,D_{x}))=\left\{ \sum_{\tiny
    \begin{array}[c]{c}
      \lambda\in \sigma(F)\\
   {\rm Im}\lambda \geq 0 
    \end{array}
} -i\lambda  (r_{\lambda}+2k_{\lambda})\,, k_{\lambda}\in \mathbb{N}\right\}\,,
$$
where $F$ is the so-called Hamilton map associated with $q_{J}$, and
$r_{\lambda}$ is the algebraic multiplicity of 
$\lambda\in \sigma(F)$\,, i.e. the dimension of the characteristic space.
The Hamilton map is the $\bC$-linear map $F:\bC^{2N}\to \bC^{2N}$
associated with the matrix
$$
F=
\begin{bmatrix}
  0&I\\
-I&0
\end{bmatrix}
M_{q_{J}}\,,
$$
where 
$$
M_{q_{J}}=
\begin{bmatrix}
  \frac{S}{4}& -\frac{i}{2}\tilde J\\
\frac{i}{2}\tilde{J}& S
\end{bmatrix}~\in \mathcal{M}_{2N}(\mathbb{C})
$$
is the matrix of the $\bC$-bilinear form associated with $q_{J}$\,.
The matrix $F$ is similar to $\tilde{F}$ defined by
$$
\tilde{F}=
\begin{bmatrix}
  \frac{1}{\sqrt{2}}&0\\
 0 & \sqrt{2}
\end{bmatrix}
F
\begin{bmatrix}
  \sqrt{2}&0\\
 0 & \frac{1}{\sqrt{2}}
\end{bmatrix}
=\frac{1}{2}
\begin{bmatrix}
  i\tilde{J}&S\\
-S & i\tilde{J}
\end{bmatrix}\,.
$$
Thus, the characteristic polynomial of $F$ can be computed by
\begin{align*}
  \det(F-\lambda I)&=\det(\tilde{F}-\lambda I)=
 2^{-2N}\left|
  \begin{array}[c]{cc}
      i\tilde{J}-2\lambda I&S\\
-S&i\tilde{J}-2\lambda I
    \end{array}
\right|
\\
&=
2^{-2N}\left|
\begin{array}[c]{cc}
      i\tilde{J}-2\lambda I&S\\
-S-\tilde{J}-i2\lambda I&i(\tilde{J}+S+i2\lambda I)
    \end{array}
\right|\\
&=
2^{-2N}\left|
\begin{array}[c]{cc}
      i(\tilde{J}-S+i2\lambda I)&S\\
0&i(\tilde{J}+S+i2\lambda I)
    \end{array}
\right|\\
&=
2^{-2N}\det(S-\tilde{J}-i2\lambda I)\det(S+\tilde{J}+i2\lambda I)\\
&=2^{-2N}\det(S+\tilde{J}-i2\lambda I)\det(S+\tilde{J}+i2\lambda I)\,,
\end{align*}
where we used $\det(M)=\det(M^{T})$ for $M=S-\tilde{J}-i2\lambda I$ in the last
line.  Using the fact that ${\rm
  Re}(\sigma (\tilde B_J)) \ge 0$, we thus obtain that $\sigma(F)\cap \left\{\lambda, \, {\rm Im} \lambda \geq
  0\right\}$ equals
$\frac{i}{2}\sigma(S+\tilde{J})=\frac{i}{2}\sigma(\tilde B_{J})$\,.
 In particular one gets,
$$
\sum_{
    \tiny\begin{array}[c]{c}
      \lambda\in \sigma(F)\\
   {\rm Im}\lambda \geq 0 
    \end{array}
} -i\lambda 2 k_\lambda= \sum_{
      \mu \in \sigma(\tilde{B}_J)}
 k_{i \mu / 2} \mu
$$
and
$$
\sum_{
    \tiny\begin{array}[c]{c}
      \lambda\in \sigma(F)\\
   {\rm Im}\lambda \geq 0 
    \end{array}
} -i\lambda r_{\lambda}=\frac{1}{2}\Tr (\tilde B_{J})=\frac{\Tr(S)}{2}\,.
$$
This concludes the proof.
\end{proof}
The Gearhart-Pr{\"u}ss theorem (see \cite{HeSj,EnNa,Trefethen_Embree}) provides the
following corollary.
\begin{corollary}
When the pair $(\tilde{J},Q)$ belongs to $\mathcal{P}_{opt}$, the
spectrum of $-\tilde{\mathcal{L}}_{J}$ is contained in
$$
\left\{0\right\}\cup\left\{z\in \bC, {\rm Re}z\geq \frac{\Tr(S)}{N}\right\}
$$
and
$$
\lim_{t\to\infty}\ln \left\|e^{t\tilde{\mathcal{L}}_{J}}(I-\Pi_{0}) \right\|_{\mathcal{L}(L^2)}=-\frac{\Tr(S)}{N},
$$
where, we recall,
$\left\|e^{t\tilde{\mathcal{L}}_{J}}(I-\Pi_{0}) \right\|_{\mathcal{L}(L^2)}=\sup_{
  u \in {\mathcal D}_0^\perp} \frac{\left\|
    e^{t\tilde{\mathcal{L}}_{J}} u \right\|_{L^2}}{\|u\|_{L^2}}$\,.
\end{corollary}
The above logarithmic convergence is weaker than an estimate
$\|e^{t\tilde{\mathcal{L}}_{J}}\|\leq Ce^{-\frac{\Tr(S)}{N}}$ with a
good control of the constant $C$. Obtaining such a control is not an
easy task for general semigroups with non self-adjoint generators (see
\cite{HeSj,HeNi04,HeNi05,GGN}). This is the subject of the next
section.

\subsection{Convergence to equilibrium for $e^{t\tilde{\mathcal L}_{J}}$}
\label{sec:returnLJ}
Consider a pair $(\tilde{J},Q)\in \mathcal{P}_{opt}$ according to
Definition~\ref{def:Popt}. We recall that $(\tilde{J},Q)\in
\mathcal{P}_{opt}$ satisfies~\eqref{e:J1}.
With the matrix $Q$, we associate the operator
\begin{equation}\label{eq:CQ}
C_{Q}=a^{*,T}Qa\,,
\end{equation}
with which a natural functional space  will be introduced  in order to study
the norm of $e^{t\tilde{\mathcal L}_{J}}$\,. The
operator $C_Q$ is the Wick-quantization of the polynomial
$\overline{z}^T Q z$.

This operator $C_{Q}$ has the
following properties: 
\begin{itemize}
\item It is continuous from $\mathcal{S}(\R^{N})$ into itself and from
  $\mathcal{S}'(\R^{N})$ into itself.
\item It is globally elliptic (see \cite{Hel,Pra1}) and it has a
  compact resolvent.
\item It is a non negative self-adjoint operator in $L^{2}(\R^{N},dx; \bC)$
  with the domain $D(C_{Q})=\left\{u\in L^{2}(\R^{N},dx; \bC),
  C_{Q}u\in L^{2}(\R^{N},dx; \bC)\right\}$\,.
\item Its kernel is $\bC e^{-\frac{|x|^{2}}{4}}$\,.
\item It is block diagonal in the decomposition
  $L^{2}(\R^{N},dx;\bC)=\bigoplus_{k\in\mathbb{N}}^{\perp}\mathcal{D}_{k}$:
  \begin{equation}
    \label{eq.orthoCQ}
\forall t\in\R\,,\quad e^{itC_{Q}}=
\bigoplus_{k\in\mathbb{N}}^{\perp}e^{it C_{Q}}
\big|_{\mathcal{D}_{k}}\,.
\end{equation}
\end{itemize}
One defines the two Hilbert spaces: 
\begin{itemize}
\item $\mathcal{H}^{1}_{Q}=\left\{u\in L^{2}(\R^{N},dx; \bC), \langle u\,,\,
    C_{Q}u\rangle_{L^2} < +\infty\right\}$\,, naturally endowed with the scalar
  product
$$
\langle u\,,\, v\rangle_{\mathcal{H}^{1}_{Q}}= \langle u\,,\,
v\rangle_{L^2}+\langle u\,,\,C_{Q}v\rangle_{L^2}\,;
$$
\item $\dot{\mathcal{H}}^{1}_{Q}=\mathcal{H}^{1}_{Q}\cap
  \mathcal{D}_{0}^{\perp}$ (where, we recall, $\mathcal{D}_{0}= \bC e^{-\frac{|x|^{2}}{4}}$) endowed with the scalar product
$$
\langle u\,,\, v\rangle_{\dot{\mathcal{H}}^{1}_{Q}}= \langle u\,,\,C_{Q}v\rangle_{L^2}\,.
$$
\end{itemize}
\begin{proposition}
\label{prop:decHQ}
Assume that the pair $(\tilde{J},Q)$ belongs to $\mathcal{P}_{opt}$\,.
 Then the semigroup $(e^{t\tilde{\mathcal{L}}_{J}})_{t\geq 0}$ is a
 contraction semigroup on $\mathcal{H}_{Q}^{1}$ satisfying the
 following estimate:
\begin{equation}\label{eq:decHQ}
\forall t\geq 0\,,\quad \left\|e^{t\tilde{\mathcal{L}}_{J}}(I-\Pi_{0})
  \right\|_{\mathcal{L}(\dot{\mathcal{H}}^{1}_{Q})}
\leq e^{-\frac{\Tr(S)}{N}t}\,,
\end{equation}
where
$\left\|e^{t\tilde{\mathcal{L}}_{J}}(I-\Pi_{0}) \right\|_{\mathcal{L}(\dot{\mathcal{H}}^{1}_{Q})}=\sup_{
  u \in \dot{\mathcal{H}}^{1}_{Q}} \frac{\left\| e^{t\tilde{\mathcal{L}}_{J}} u \right\|_{\dot{\mathcal{H}}^{1}_{Q}}}{\|u\|_{\dot{\mathcal{H}}^{1}_{Q}}}$\,.
\end{proposition}
\begin{proof}
The operator $e^{t\tilde{\mathcal{L}}_{j}}$ is block diagonal (see
Proposition~\ref{prop:semigtLJ})
 in the
decomposition
$\mathcal{D}_{0}\mathop{\oplus}^{\perp}\mathcal{D}_{0}^{\perp}=\bigoplus_{k\in\mathbb{N}}^{\perp}\mathcal{D}_{k}$
which is an orthogonal decomposition
 in $L^{2}(\R^{N},dx;\bC)$ and also
in $\mathcal{H}^{1}_{Q}$ owing to \eqref{eq.orthoCQ}. With
$e^{t \tilde{\mathcal{L}}_{J}}e^{-\frac{|x|^{2}}{4}}=e^{-\frac{|x|^{2}}{4}}$\,, 
the semigroup property on $\mathcal{H}^{1}$ is thus a consequence
of the estimate~\eqref{eq:decHQ} in $\dot{\mathcal{H}}^{1}_{Q}$\,.\\
Using the relation~\eqref{e:J1} together with the
inequality~\eqref{eq:ineqexample} of Lemma~\ref{lem:example} in the
Appendix, we have: for all $u\in
\mathcal{D}=\bC[x_{1},\ldots,x_{N}]e^{-\frac{|x|^{2}}{4}}\cap \mathcal{D}_{0}^{\perp}$,
$$
\left\langle u\,,\,
\left(-\tilde{\mathcal{L}}_{J}^{*}C_{Q}-C_{Q}\tilde{\mathcal{L}}_{J}\right)
u \right\rangle_{L^2}
\geq \frac{2\Tr(S)}{N}\langle u\,,\,C_{Q}u \rangle_{L^2}.
$$
 Since
the semigroup $(e^{t\tilde{\mathcal{L}}_{J}})_{t\geq 0}$ is a strongly
$\mathcal{C}^{1}$ semigroup on 
$\mathcal{S}(\R^{N})$ and leaves
$\mathcal{D}\subset\mathcal{S}(\R^{N})$ invariant, 
we can compute for any  $u\in \mathcal{D}$\,,
\begin{eqnarray*}
  \frac{d}{dt} \left\langle e^{t\tilde{\mathcal{L}}_{J}}u\,,\,
  C_{Q}e^{t\tilde{\mathcal{L}}_{J}}u \right\rangle_{L^2}
&=& \left\langle e^{t\tilde{\mathcal{L}}_{J}}u\,,\,
   \left(\tilde{\mathcal{L}}_{J}^{*}C_{Q}+C_{Q}\tilde{\mathcal{L}}_{J}\right)
     e^{t\tilde{\mathcal{L}}_{J}}u \right\rangle_{L^2}
\\
&\leq & -\frac{2\Tr(S)}{N} \left\langle e^{t\tilde{\mathcal{L}}_{J}}u\,,\,
 C_{Q}e^{t\tilde{\mathcal{L}}_{J}}u \right\rangle_{L^2}.
\end{eqnarray*}
The proof is then completed using the density of $\mathcal{D}$ in $\dot{\mathcal{H}}^{1}_{Q}$\,.
\end{proof}
We are now in position to state the main result of this section.
\begin{proposition}
\label{prop:return1}
Assume that the pair $(\tilde{J},Q)$ belongs to $\mathcal{P}_{opt}$\,. Then the semigroup
 $(e^{t\tilde{\mathcal{L}}_{J}})_{t\geq 0}$ satisfies:
 \begin{multline*}
  \forall t\geq 0\,,\quad
\left\|e^{t\tilde{\mathcal{L}}_{J}}(I-\Pi_{0})
\right\|_{\mathcal{L}(L^2)} \leq 2^{5}N\kappa(Q)^{1/2}
\\\times\left(\frac{\max\sigma(Q)}{\min_{\lambda,\lambda'\in
    \sigma(Q)\,,\,\lambda\neq
    \lambda'}|\lambda-\lambda'|}\right)^{2}\kappa(S)^{7/2}
e^{-\frac{\Tr(S)}{N}t}\,.
\end{multline*}
\end{proposition}
\begin{proof}
  From the inequalities on real symmetric matrices$\min\sigma(Q)
  \,I\leq Q\leq \max \sigma(Q) \,I$ and $\min \sigma(S) \, I\leq S\leq \max \sigma(S) \,I$,
we deduce with the help of Proposition~\ref{prop:Wickcalcul}-1  the
following inequalities on self-adjoint operators
\begin{eqnarray*}
  && \min \sigma(Q)a^{*,T}a\leq C_{Q}\leq \max\sigma(Q)a^{*,T}a\,,\\
&&\min \sigma(S)a^{*,T}a\leq
-\frac{\tilde{\mathcal{L}}_{J}+\tilde{\mathcal{L}}^{*}_{J}}{2}
\leq
\max\sigma(S)a^{*,T}a\,,
\\
\text{and}&&
-\frac{\min\sigma(Q)}{\max\sigma(S)}\frac{\tilde{\mathcal{L}}_{J}+\tilde{\mathcal{L}}^{*}_{J}}{2}
\leq C_{Q}\leq -\frac{\max\sigma(Q)}{\min \sigma(S)}
\frac{\tilde{\mathcal{L}}_{J}+\tilde{\mathcal{L}}^{*}_{J}}{2}\,.
\end{eqnarray*}
Here, we have used the fact that
$-\frac{\tilde{\mathcal{L}}_{J}+\tilde{\mathcal{L}}^{*}_{J}}{2}$ is
the Wick quantization of $( z, S z )_{\bC}$ (see
Proposition~\ref{prop:Wickcalcul} and Lemma~\ref{lem:example} below).
Hence, using Proposition~\ref{prop:decHQ}, the following
inequalities hold: for any $u\in
\dot{\mathcal{H}}^{1}_{Q}$ and any $t\ge t_{0}>0$,
\begin{eqnarray*}
 && \left\langle e^{t\tilde{\mathcal{L}}_{J}}u\,,\,
  -\frac{\tilde{\mathcal{L}}_{J}+\tilde{\mathcal{L}}_{J}^{*}}{2} 
e^{t\tilde{\mathcal{L}}_{J}}u \right\rangle_{L^2}
\leq  
\frac{\max \sigma(S)}{\min
  \sigma(Q)} \left\|e^{t\tilde{\mathcal{L}}_{J}}u \right\|_{\dot{\mathcal{H}}^{1}_{Q}}^{2}
\\
&&\qquad\leq  \frac{\max \sigma(S)}{\min
  \sigma(Q)}e^{-2\frac{\Tr(S)}{N}(t-t_{0})}
\left\|e^{t_{0}\tilde{\mathcal{L}}_{J}}u \right\|_{\dot{\mathcal{H}}^{1}_{Q}}^{2}
\\
&&\qquad\leq 
\kappa(Q)\kappa(S)e^{-2\frac{\Tr(S)}{N}t}e^{2\frac{\Tr(S)}{N}t_{0}}
\left\langle e^{t_{0}\tilde{\mathcal{L}}_{J}}u\,,\,
  -\frac{\tilde{\mathcal{L}}_{J}+\tilde{\mathcal{L}}_{J}^{*}}{2} 
e^{t_{0}\tilde{\mathcal{L}}_{J}}u\right\rangle_{L^2}
\\
&&
\qquad 
\leq
\kappa(Q)\kappa(S)e^{-2\frac{\Tr(S)}{N}t}e^{2\frac{\Tr(S)}{N}t_{0}}
\left\|e^{t_{0}\tilde{\mathcal{L}}_{J}}u\right\|_{L^2}
\left\|\tilde{\mathcal{L}}_{J}e^{t_{0}\tilde{\mathcal{L}}_{J}}u \right\|_{L^2}.
\end{eqnarray*}
Using the inequalities
 \begin{equation*}
\forall v\in \mathcal{D}_{0}^{\perp},\quad
   \min\sigma(S)\|v\|_{L^2}^{2}\leq \left\langle v\,,\,
  -\frac{\tilde{\mathcal{L}}_{J}+\tilde{\mathcal{L}}_{J}^{*}}{2}
  v \right\rangle_{L^2}
\leq \|v\|_{L^2} \left\|\tilde{\mathcal{L}}_{J}v \right\|_{L^2},
 \end{equation*}
with $v=e^{t\tilde{\mathcal{L}}_{J}}u$ and 
$v=e^{t_{0}\tilde{\mathcal{L}}_{J}}u$, we deduce
\begin{equation*}
  \|e^{t\tilde{\mathcal{L}}_{J}}u\|_{L^2}^{2}\leq \kappa(Q)\kappa(S) 
e^{-2\frac{\Tr(S)}{N}t}\frac{e^{2\frac{\Tr(S)}{N}t_{0}}}{t_{0}^{2}\min
\sigma(S)^{2}}
\left\|t_{0}\tilde{\mathcal{L}}_{J}e^{t_{0}\tilde{\mathcal{L}}_{J}}u \right\|_{L^2}^{2}.
\end{equation*}
By taking $t_{0}=\frac{N}{\Tr(S)}\geq \frac{1}{\max \sigma(S)}$, we
obtain, for all $u\in \dot{\mathcal{H}}^{1}_{Q}$\,,
\begin{equation}\label{eq:11}
 \|e^{t\tilde{\mathcal{L}}_{J}}u\|_{L^2}^{2}\leq \kappa(Q)\kappa(S)^{3}
e^{-2\frac{\Tr(S)}{N}t} e^2
\sup_{t'>0}
\left\|t'\tilde{\mathcal{L}}_{J}e^{t'\tilde{\mathcal{L}}_{J}}\right\|_{\mathcal{L}(L^2)}^{2}
\|u\|_{L^2}^{2}\,.
\end{equation}
The Lemma~\ref{lem:sector} below provides the bound
\begin{equation}\label{eq:tLetL}
\sup_{t'>0}
\left\|t'\tilde{\mathcal{L}}_{J}e^{t'\tilde{\mathcal{L}}_{J}} \right\|_{\mathcal{L}(L^2)}^{2}
\leq \frac{1}{\pi^{2}\sin^{4}\alpha}
\end{equation}
with $\alpha \in (0,\pi/4)$ defined by
$$\tan\left(\frac{\pi}{2}-2\alpha\right)=\sup_{u \in D(\tilde{\mathcal L}_J)}\frac{|{\rm Im}~\langle u\,,\,
  \tilde{\mathcal{L}}_{J}u\rangle_{L^2}|}{|{\rm Re}~\langle u\,,\,
  \tilde{\mathcal{L}}_{J}u\rangle_{L^2}|} \leq \frac{\|\tilde{J}\|}{\min
  \sigma(S)}\,.
$$
The last inequality was proven in~\eqref{eq:sectorial} above. We thus
obtain
$$
\frac{1}{\sin \alpha}\leq
\frac{2\cos\alpha}{\cos(2\alpha)}\frac{\|\tilde{J}\|}{\min \sigma(S)}\,.
$$
In view of~\eqref{eq:tLetL}, one can assume that $\alpha \in (0,\pi/8)$ (up to
changing $\alpha$ by $\min(\alpha, \pi/8)$) so that
\begin{equation}\label{eq:22}
\frac{1}{\sin \alpha}\leq 2\sqrt{2}\frac{\|\tilde{J}\|}{\min \sigma(S)}\,.
\end{equation}
When $(\tilde{J},Q)\in \mathcal{P}_{opt}$ (see
Definition~\ref{def:Popt}), the relation~\eqref{e:Jsoln}
provides an expression of the linear mapping associated with
$\tilde{J}$ in the orthonormal basis $(\psi_{k})_{1\leq k\leq N}$\,.
In this basis, the Frobenius norm can be computed and we get
\begin{align}
\| \tilde J \|^2
& \le \| \tilde J \|^2_F 
\nonumber
\\
& \le 2 \left(\frac{\max \sigma(Q)}{\min_{\lambda,\lambda'\in
    \sigma(Q)\,,\,\lambda\neq \lambda'}|\lambda-\lambda'|} \right)^2
\| S \|^2_F 
\nonumber
\\
&\leq  2 \left(\frac{\max \sigma(Q)}{\min_{\lambda,\lambda'\in
    \sigma(Q)\,,\,\lambda\neq \lambda'}|\lambda-\lambda'|} \right)^2
N\max(\sigma(S))^{2}\,,
\label{eq:33}
\end{align}
By gathering~\eqref{eq:11}--\eqref{eq:tLetL}--\eqref{eq:22}--\eqref{eq:33},
we finally obtain the expected upper bound when $t \ge t_{0}$\,:
\begin{align*}
\|e^{t\tilde{\mathcal{L}}_{J}}u\|_{L^2}^{2}&\leq 2^{10}N^{2}
\kappa(Q)\left(\frac{\max\sigma(Q)}{\min_{\lambda,\lambda'\in
    \sigma(Q)\,,\,\lambda\neq
    \lambda'}|\lambda-\lambda'|}\right)^{4}
\\
&\quad \times\kappa(S)^{7}
e^{-2\frac{\Tr(S)}{N}t}
\|u\|_{L^2}^{2}\,,
\end{align*}
 for all $u\in \dot{\mathcal{H}}^{1}_{Q}$ and by density for all $u\in
 \mathcal{D}_{0}^{\perp}$\,.
When $t\leq t_{0}=\frac{N}{\Tr S}$\,, simply use
\begin{eqnarray*}
&&\|e^{t\tilde{\mathcal{L}}_{J}}(I-\Pi_{0})\|_{\mathcal{L}(L^{2})}
\leq 1
\\
&&\quad\leq
2^{5}N
\kappa(Q)^{\frac 1 2}\left(\frac{\max\sigma(Q)}{\min_{\lambda,\lambda'\in
    \sigma(Q)\,,\,\lambda\neq
    \lambda'}|\lambda-\lambda'|}\right)^{2}
\kappa(S)^{\frac 7 2}\times e^{-1}\,.
\end{eqnarray*}
\end{proof}

\begin{remark}\label{rem:bornJ}
A lower bound can be given for $\|\tilde{J}\|$ with
\begin{align*}
\| \tilde J \|^2
& \ge \frac{1}{N} \| \tilde J \|^2_F \\
& \ge \frac2N \left(\frac{\min \sigma(Q)}{\max_{\lambda,\lambda'\in
    \sigma(Q)\,,\,\lambda\neq \lambda'}|\lambda-\lambda'|} \right)^2
\| S \|^2_F \\
&=  \frac2N \left(\frac{\min \sigma(Q)}{\max_{\lambda,\lambda'\in
    \sigma(Q)\,,\,\lambda\neq \lambda'}|\lambda-\lambda'|} \right)^2
\Tr (S^2).
\end{align*}
Thus, we have
\begin{equation}\label{eq:born_inf_J}
\|\tilde{J}\|\geq \sqrt{2}  \frac{\min\sigma(Q)\min\sigma(S)}{\max_{\lambda,\lambda'\in
    \sigma(Q)\,,\,\lambda\neq \lambda'}|\lambda-\lambda'|}\,.
\end{equation}
\end{remark}
\begin{lemma}
\label{lem:sector}
  Let $(L,D(L))$ be a maximal accretive and sectorial operator  in a Hilbert
  space $\mathcal{H}$ with
$$
\forall u\in D(L)\,,\quad
\left|{\rm arg}\langle u\,,\, Lu\rangle_{\mathcal{H}}\right|\leq \theta =\frac{\pi}{2}-2\alpha
\quad\text{with}~\alpha>0\,,
$$
where, we recall, ${\rm arg}(z)$ denotes the argument
of a complex number $z$\,.
Then, the associated semigroup satisfies
$$
\forall t\geq 0\,,\quad
\left\|tLe^{-tL}\right\|_{\mathcal{L}(\mathcal{H})}\leq
\frac{1}{\pi\sin^{2}\alpha}\,.
$$
\end{lemma}
\begin{proof}
The case $t=0$ is obvious.\\
For $t>0$\,, $e^{-tL}$ sends $\mathcal{H}$ into $D(L)$ so that
$tLe^{-tL}$ belongs to $\mathcal{L}(\mathcal{H})$\,.
Consider first the case when $0\not\in\sigma(L)$\,.
Our assumptions
with $\alpha>0$\,, ensure that the operator $tLe^{-tL}$ is given by
the convergent contour integral 
$$
tLe^{-tL}=\frac{1}{2i\pi}\int_{\Gamma}tze^{-tz}(z-L)^{-1}~dz\,,
$$
where $\Gamma$ is the union of the two half lines with arguments
$\frac{\pi}{2}-\alpha$ and $-\frac{\pi}{2}+\alpha$\,. For
$z=x\pm i\frac{x}{\tan\alpha}\in \Gamma$ with $x>0$ the resolvent 
$(z-L)^{-1}$ satisfies (see for example \cite[Chapter~VIII.17]{ReSi})
$\|(z-L)^{-1}\|_{\mathcal{L}(\mathcal H)}\leq \frac{1}{x}$\,. Moreover,
$|dz|=\sqrt{1+\frac{1}{\tan^{2} \alpha}} dx=\frac{dx}{\sin
  \alpha}$  and $|e^{-tz}|=e^{-tx}$\,.
From these estimates, we deduce
\begin{align*}
\left\|tLe^{-tL}\right\|_{\mathcal H}\leq \frac{2}{2\pi}\int_{0}^{+\infty}\frac{tx}{\sin
  \alpha}e^{-tx}\frac{1}{x}\frac{dx}{\sin \alpha}
&=\frac{1}{\pi\sin^{2}
  \alpha}
\int_{0}^{\infty} t e^{-tx} \, dx
\\
&=\frac{1}{\pi\sin^{2}
  \alpha}\,.
\end{align*}
When $0\in \sigma(L)$ it suffices to replace $L$ by $\varepsilon+L$
which satisfies the same assumptions as $L$ with the same $\alpha$
with $0\not\in\sigma(\varepsilon+L)$\,. The identity
$$
t(\varepsilon+L)e^{-t(\varepsilon+L)}-Le^{-tL}
=t\varepsilon e^{-\varepsilon t}e^{-tL}+(e^{-\varepsilon t}-1)tLe^{-tL}
$$
with $t>0$ fixed and $e^{-tL}\,,\,
tLe^{-tL}\in\mathcal{L}(\mathcal{H})$ implies 
$\lim_{\varepsilon\to
  0}\|t(\varepsilon+L)e^{-t(\varepsilon+L)}-tLe^{-tL}\|_{\mathcal{L}(\mathcal{H})}=0$\,,
which yields the result in the general case.
\end{proof}

In view of~\eqref{eq:CV_KB'}, Proposition~\ref{prop:return1} yields
the estimate~\eqref{eq:CV_KB} in
Theorem~\ref{th:CV_KB} with a constant
\begin{equation}\label{eq:CN_PDE}
C_N^{(2)}=2^{5}N\kappa(Q)^{1/2}\left(\frac{\max\sigma(Q)}{\min_{\lambda,\lambda'\in
    \sigma(Q)\,,\,\lambda\neq
    \lambda'}|\lambda-\lambda'|}\right)^{2}.
\end{equation}
To conclude, let us comment on the way of $C_N$ behaves.
\begin{remark}\label{rem:CN_PDE}
In view of the upper bound~\eqref{eq:CN_PDE}, using the same
construction as in
 Remark~\ref{rem:CN_ODE},  we again
notice that it is possible to have $C_N^{(2)}=\mathcal{O}(N^{3})$ while keeping
a reasonable perturbation $\tilde J$ (with a Frobenius norm estimated
by $\|\tilde{J}\|_{F}\leq 4N\|S\|_{F}$). Contrary to the case of
the ordinary differential equation discussed in Remark~\ref{rem:CN_ODE} our estimate does not
provide a uniform in~$N$ constant.
\end{remark}

%
%


%
%

\section{Numerical Experiments}
\label{sec:num}

The algorithm for obtaining the optimal non-reversible is presented as a pseudo-code in Figure~\ref{fig:algor}.

In this section we present some numerical experiments, based on the algorithm presented in Figure~\ref{fig:algor}. The numerical computations presented in this section are based on the following steps:

\begin{enumerate}

\item Calculate the orthonormal basis $\{\psi_{k}\}_{k=1}^{N}$ using the algorithm presented in Figure~\ref{fig:algor}.

\item Choose the eigenvalues of the matrix $Q$, $\{\lambda_{k} \}_{k=1}^{N}$, e.g. according to Remark~\ref{rem:CN_ODE}.

\item Calculate the optimal perturbation $J$ using~\eqref{e:Jsoln} and the formula $J = S^{-1/2} \tilde{J} S^{-1/2}$.

\item Calculate the optimally perturbed matrix $B_{J} = (I + J) S$.

\item Calculate the matrix exponentials $e^{-tB}$ and $e^{-t B_{J}}$ and their norms.

\end{enumerate}

\begin{figure}[!ht]
\begin{center}
\includegraphics[width=1.1\textwidth]{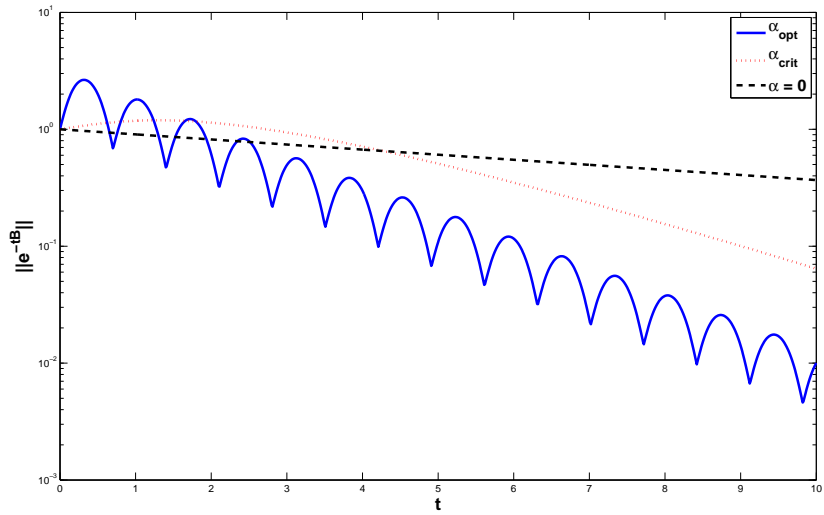}
\end{center}
\caption{Norms of the matrix exponentials for the $2 \times 2$ diagonal matrix~\eqref{e:2dexamp} and optimal nonreversible perturbations.} \label{fig:matrixepon2d}
\end{figure}

In Figure~\ref{fig:matrixepon2d} we present the results for a two dimensional problem, for which all results can be performed analytically, see Section~\ref{sec:2D}. We consider the case where the matrix $B$ has a spectral gap, 
\begin{equation}\label{e:2dexamp}
S = \mbox{diag}(1,0.1).
\end{equation} 
In the figure we plot the norms of the matrix exponentials for the symmetric case, an optimal perturbation and the critical value, see Equation~\eqref{eq:cond_a}.

\begin{figure}[!ht]
\begin{center}
\includegraphics[width=1.1\textwidth]{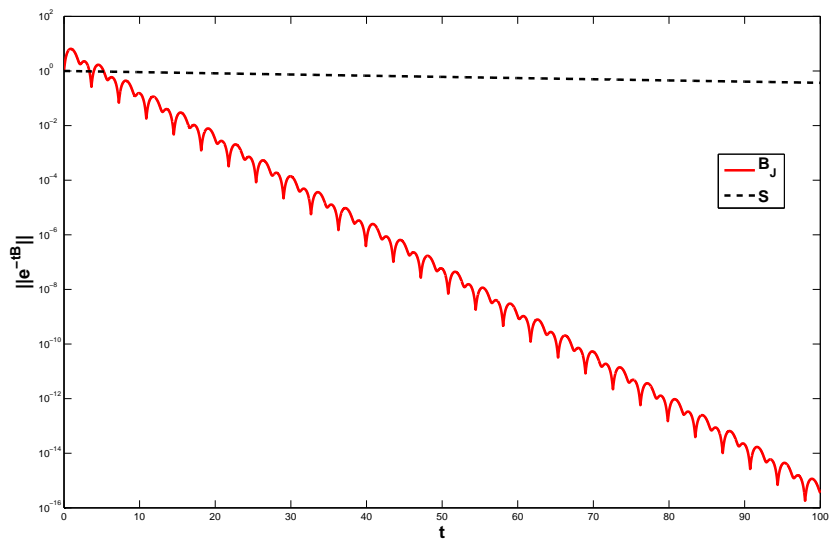}
\end{center}
\caption{Norms of the matrix exponentials for the $3 \times 3$ diagonal matrix~\eqref{e:3dexamp} and its optimal nonreversible perturbation.} \label{fig:matrixexpon3d}
\end{figure}

In Figure~\ref{fig:matrixexpon3d} we present results for a three dimensional problem with the symmetric matrix
\begin{equation}\label{e:3dexamp}
S = \mbox{diag}(1,\, 0.1,\, 0.01).
\end{equation}
The spectral gap of the optimally perturbed nonreversible matrix (and of the generator of the semigroup) is given by
$$
\frac{\Tr{S}}{3} = 0.37,
$$
which is a substantial improvement over that of $S$, namely $0.01$.

\begin{figure}[!ht]
\begin{center}
\includegraphics[width=1.1\textwidth]{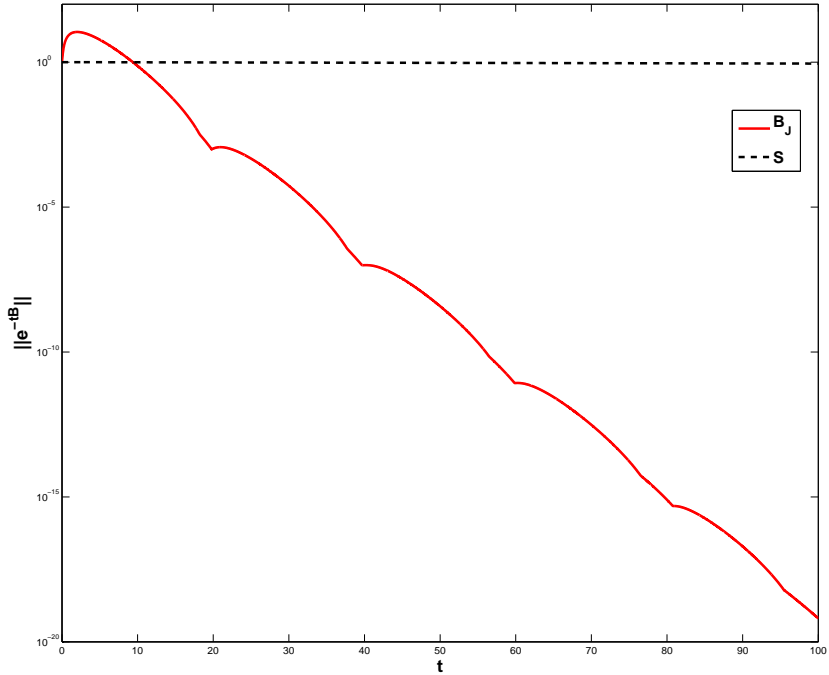}
\end{center}
\caption{Norms of the matrix exponentials for a diagonal matrix with random uniformly distributed entries and its optimal nonreversible perturbation for $N=100$.} \label{fig:matrixexponrand100}
\end{figure}

In Figure~\ref{fig:matrixexponrand100} we consider a $100 \times 100$ diagonal matrix with random entries, uniformly distributed on $[0,1]$. For our example the minimum diagonal element (spectral gap) is $0.0012$. On the contrary, the spectral gap of $B_{J}$ with $J = J_{opt}$ is $0.4762$.

\begin{figure}[!ht]
\begin{center}
\includegraphics[width=1.1\textwidth]{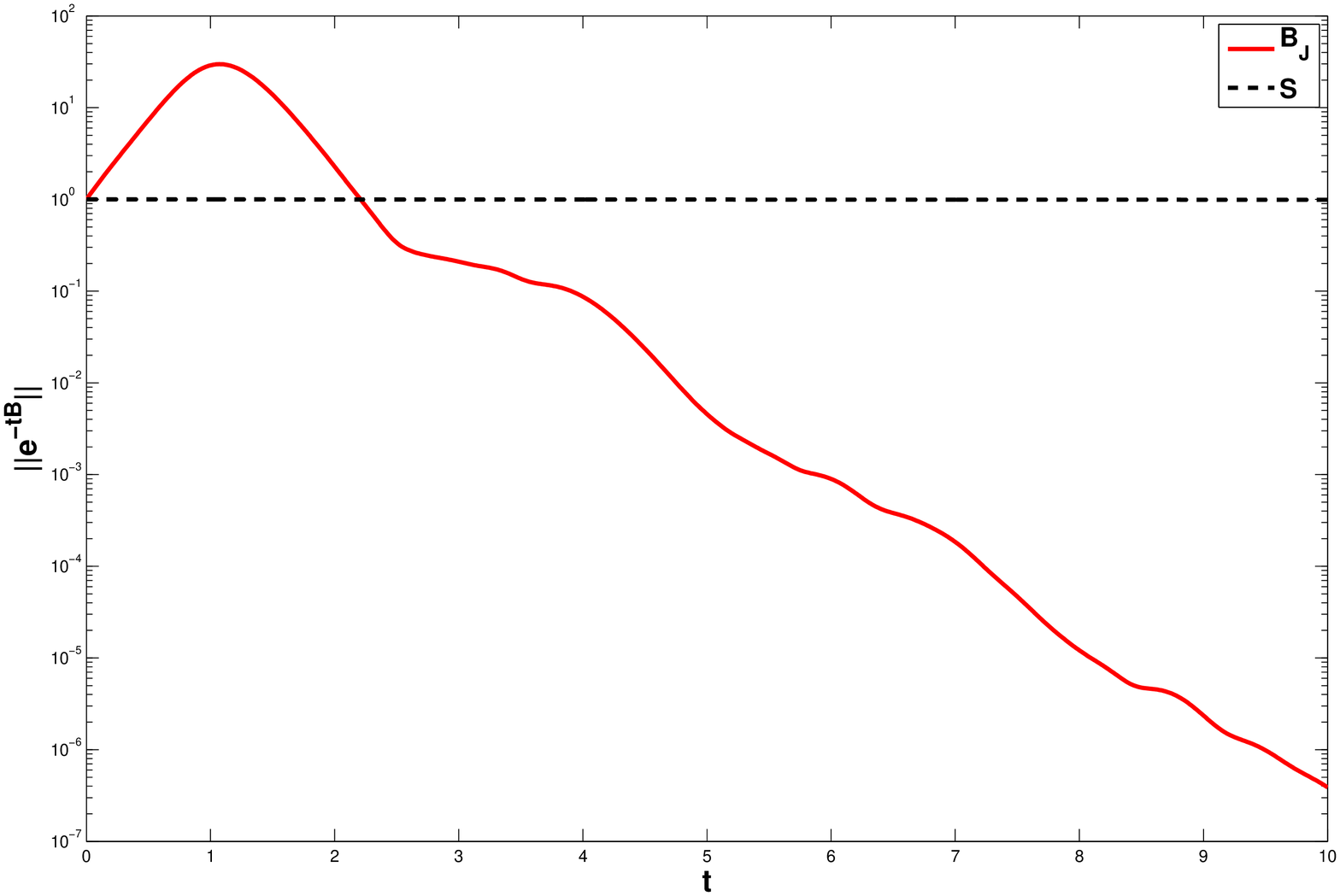}
\end{center}
\caption{Norms of the matrix exponentials for the the discrete Laplacian and its optimal nonreversible perturbation for $N=100$.} \label{fig:matrixexpondiscrlapl100}
\end{figure}

Finally, in Figure~\ref{fig:matrixexpondiscrlapl100} we consider a
drift that is a (high dimensional) finite difference approximation of
of the Laplacian with periodic boundary conditions. More precisely, consider the drift matrix
$$
B_{ii} = 2, \quad B_{i,i+1} = B_{i-1,i} = -1,
$$ 
with $N=100$. In this case the improvement on the convergence rate is over three orders of magnitude, since 
$$
\min(\sigma(B)) = 9.67 \times 10^{-4}, \quad \mbox{whereas} \quad \mbox{Re}(\sigma({B_{J}})) = \frac{\Tr{S}}{100} = 2.
$$
Since the computational cost of calculating the optimal nonreversible perturbation is very low, we believe that the algorithm developed in this paper can be used for sampling Gaussian distributions in infinite dimensions.

The algorithm developed in this paper provides us with the optimal
nonreversible perturbation only in the case of linear drift. However,
even for nonlinear problems it is always the case that the addition of
a nonreversible perturbation can accelerate the convergence to
equilibrium, as mentioned in the introduction. This is particularly the case for systems with metastable states and/or multiscale structure~\cite{lelievre-12}; for such systems, a ``clever'' choice of the nonreversible perturbation can lead to a very significant increase in the rate of convergence to equilibrium. A systematic methodology for obtaining the optimal nonreversible perturbation for general reversible diffusions (i.e. not necessarily with a linear drift) will be developed elsewhere.

We illustrate the advantage of adding a nonreversible perturbation to the dynamics by considering a few simple two-dimensional examples. In particular, we consider the nonreversible dynamics
\begin{equation}\label{e:schuette2d}
d X_{t} = (-I+\delta J)\nabla V(X_{t}) \, dt + \sqrt{2 \beta^{-1}} \, dW_{t},
\end{equation}
with $\delta \in \R$ and $J$ the standard $2 \times 2$ antisymmetric
matrix, i.e. $J_{12} = 1, \, J_{21} = -1$. For this class of
nonreversible perturbations the parameter that we wish to choose in an
optimal way is $\delta$. From our numerical experiments, we observed
that even a non-optimal choice of $\delta$ significantly accelerates convergence to equilibrium. To illustrate the effect of adding a nonreversible perturbation, we solve numerically~\eqref{e:schuette2d} using the Euler-Marayama method with a sufficiently small time step and for a sufficiently large number of realizations of the noise. We then compute the expectation value of observables of the solution, in particular, the second moment by averaging over all the trajectories that we have generated.

We  use one of the potentials that were considered
in~\cite{metznerSchutteEijnden06}, namely
\begin{equation}\label{e:potential-ve}
V(x,y) = \frac{1}{4}(x^{2} - 1)^{2} + \frac{1}{2} y^{2}.
\end{equation}
\begin{figure}[!ht]
\begin{center}
\includegraphics[width=1.1\textwidth]{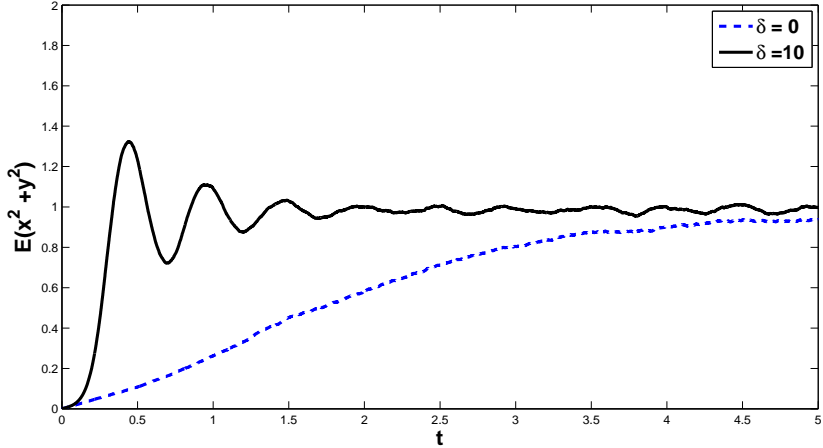}
\end{center}
\caption{Second moment as a function of time for~\eqref{e:schuette2d}
  with the potential~\eqref{e:potential-ve}. We take $0$ as an initial condition and $\beta^{-1} = 0.1$.} \label{fig:2dnonlinearSchuette1}
\end{figure}

In Figure~\ref{fig:2dnonlinearSchuette1} we present the convergence of
the second moment to its equilibrium value for $\beta^{-1} =
0.1$. Even in this very simple example, the addition of a
nonreversible perturbation, with $\delta = 10$, speeds up convergence
to equilibrium. Notice also that, as expected, the nonreversible
perturbation leads to an oscillatory transient behavior.

\appendix{}

\section{Wick calculus}
\label{sec:wickcalculus}

In this article, we use a specific positivity property of the Wick
calculus, which must not be confused with the more general and robust
 positivity
property of the anti-Wick calculus
This appendix recalls the basic
facts about Wick calculus and its positivity property in
$L^{2}(\mathbb{R}^{N}, dx; \mathbb{C})$\,. We refer the reader for details
 to \cite{AmNi1,AmNi2} and references therein. This calculus is modelled on
the creation and annihilation operators,
$a_j^*=-\partial_{x_{j}}+\frac{1}{2}x_{j}$ and
$a_j=\partial_{x_{j}}+\frac{1}{2}x_{j}$, of the $N$-dimensional harmonic
oscillator Hamiltonian
$-\Delta_{x}+\frac{|x|^{2}}{4}-\frac{N}{2}=\sum_{j=1}^{N}a_{j}^{*}a_{j}$\,.
The kernel of this Hamiltonian (the so-called vacuum state) is the
Gaussian function
$$
|\Omega\rangle=\frac{1}{(2\pi)^{N/4}}e^{-\frac{|x|^{2}}{4}}\,.
$$
Let $(e_{1},\ldots,e_{N})$ be an orthonormal basis of
$\mathbb{C}^{N}$ and use the notation
$$
a(e_{j})=\partial_{x_{j}}+\frac{1}{2}x_{j},\quad
a^{*}(e_{j})=-\partial_{x_{j}}+\frac{1}{2}x_{j}\,.
$$
The canonical commutation relations are (for two operators $A$ and
$B$, we denote $[A,B]=AB-BA$ the commutator bracket)
$$
\left[a(e_{i}),
a(e_{j})\right]=\left[a^{*}(e_{i}),a^{*}(e_{j})\right]=0,\,
\quad [a(e_{i}), a^{*}(e_{j})]=\delta_{i,j}
$$
and become by sesquilinearity of the complex scalar product
$$
\forall
z_{1},z_{2}\in\mathbb{C}^{N}\,,\quad\left[a(z_{1}),a^{*}(z_{2})\right]=(
z_{1}\,,\, z_{2})_{\bC}\,,
$$
after setting for $z\in\mathbb{C}^{N}$, $z=\sum_{j=1}^{N}z^{j}e_{j}$
$$
a(z)=\sum_{j=1}^{N}\overline{z^{j}}a(e_{j}),\quad
a^{*}(z)=\sum_{j=1}^{N}z^{j}a^{*}(e_{j})\,.
$$
We use the multi-index notation introduced in Section~\ref{sec:propsemi}.
The orthonormal basis of Hermite functions in
$L^{2}(\R^{N}, dx;\mathbb{C})$ is $(\phi_{n})_{n\in\mathbb{N}^{N}}$ given
by
$$
\phi_{n}=\frac{1}{\sqrt{n_{1}!\ldots
n_{N}!}}a^{*}(e_{1})^{n_{1}}\ldots a^{*}(e_{N})^{n_{N}}|\Omega\rangle
=\frac{1}{\sqrt{n!}}(a^{*}(e))^{n}|\Omega\rangle\,.
$$
Since the basis of
 the subspace $\bigoplus_{|n|=p}\mathbb{C}\phi_{n}$ is by construction
 indexed by the 
 $n$'s in $\mathbb{N}^{N}$ such that $|n|=p$\,, it can be
considered as the $p$-fold symmetric tensor product of $\mathbb{C}^{N}$\,:
\begin{eqnarray*} && \bigoplus_{|n|=p}\mathbb{C}\phi_{n} =
\bigvee^{p}\mathbb{C}^{N}=\mathcal{S}_{p}
\underbrace{\left[\mathbb{C}^{N}\otimes\cdots \otimes
\mathbb{C}^{N}\right]}_{p~\text{times}}\,, \\ \text{with}&&
\mathcal{S}_{p}(z_{1}\otimes\cdots\otimes z_{p})=\frac{1}{p!}
\sum_{\sigma\in \mathfrak{S}_{p}}z_{\sigma(1)}\otimes\cdots\otimes
z_{\sigma(p)}\,.
\end{eqnarray*} The symmetrization operator $\mathcal{S}_{p}$ is
actually the orthogonal projection from $(\mathbb{C}^{N})^{\otimes p}$
onto $\bigvee^{p}\mathbb{C}^{N}$\,.  This provides the description of
$L^{2}(\R^{N}, dx;\mathbb{C})$ as the so-called bosonic Fock space over
$\mathbb{C}^{N}$
$$
L^{2}(\R^{N}, dx;\mathbb{C})=\bigoplus_{p=0}^{\infty}\bigvee^{p}\mathbb{C}^{N}\,,
$$
where the infinite sum is orthogonal and complete.  It is convenient
to introduce also the algebraic orthogonal sum
$$
\mathcal{D}=\bigoplus_{p\in\mathbb{N}}^{alg}\bigvee^{p}\mathbb{C}^{N}=\mathbb{C}[x_{1},\ldots,x_{N}]e^{-\frac{|x|^{2}}{4}}\,.
$$
We now consider polynomials of $2N$ real coordinates $(x,y)$ with $z=x+iy$\,, written in the
complex notation $(z,\overline{z})$ as elements of
$$
\mathbb{C}[\overline{z^{1}},\ldots, \overline{z^{N}},z^{1},\ldots,
z^{N}]=\bigoplus_{(p,q)\in\mathbb{N}^{2}}^{alg}
\mathbb{C}_{q,p}[\overline{z^{1}},\ldots,
\overline{z^{N}},z^{1},\ldots, z^{N}]\,,
$$ 
where $\mathbb{C}_{q,p}[\overline{z^{1}},\ldots,
\overline{z^{N}},z^{1},\ldots, z^{N}]$ denote the set of monomials homogeneous with degree
$q\in \mathbb{N}$ with respect to $\overline{z}$ and homogeneous with degree $p \in \mathbb{N}$
with respect to $z$. Notice that monomials $b\in \mathbb{C}_{q,p}[\overline{z^{1}},\ldots,
\overline{z^{N}},z^{1},\ldots, z^{N}]$
can be written
$$
b(z)=( z^{\otimes q}\,,\, \tilde{b}z^{\otimes
p})_{\bC} \quad\text{with}\quad \tilde{b}\in
L\left(\bigvee^{p}\mathbb{C}^{N};\bigvee^{q}\mathbb{C}^{N} \right)\,.
$$
This provides a bijection between the sets
$\mathbb{C}_{q,p}[\overline{z^{1}},\ldots,
\overline{z^{N}},z^{1},\ldots, z^{N}]$ and
$L(\bigvee^{p}\mathbb{C}^{N};\bigvee^{q}\mathbb{C}^{N})$ with the
inversion formula
$$
\tilde{b}=\frac{1}{q!p!}\partial_{\overline{z}}^{q}
\, \partial_{z}^{p} \, b\,.
$$
With any monomial $b\in \mathbb{C}_{q,p}[\overline{z^{1}},\ldots,
\overline{z^{N}},z^{1},\ldots, z^{N}]$ (and by linearity with any
polynomial of $(z,\overline{z})$) we can associate an operator
$b^{Wick}:\mathcal{D}\to \mathcal{D}\subset L^{2}(\R^{N}, dx;\mathbb{C})$
called its Wick quantization:\\
 When $b\in
\mathbb{C}_{q,p}[\overline{z^{1}},\ldots,
\overline{z^{N}},z^{1},\ldots, z^{N}]$, and for any $n\in\mathbb{N}$
its restriction 
 $b^{Wick}\big|_{\bigvee^{n}\mathbb{C}^{N}} :
\bigvee^{n}\mathbb{C}^{N}\to \bigvee^{n+q-p}\mathbb{C}^{n+q-p}$ is
defined by
\begin{equation}
  \label{eq:defWick}
b^{Wick}\big|_{\bigvee^{n}\mathbb{C}^{N}}=1_{[p,+\infty)}(n)\frac{\sqrt{n!(n+q-p)!}}{(n-p)!}\mathcal{S}_{n+q-p}
(\tilde{b}\otimes I_{\bigvee^{n-p}\mathbb{C}^{N}})\mathcal{S}_{n}\,.
\end{equation}
Here are a few examples
\begin{itemize}
\item if $b(z)=( \xi \,,\,z)_{\bC}$\,, $\xi \in
\mathbb{C}^{N}$\,, then $b^{Wick}=a(\xi)$\,;
\item if $b(z)=( z \, , \, \zeta )_{\bC}$, $\zeta\in \mathbb{C}^{N}$\,,
then $b^{Wick}=a^{*}(\zeta)$\,; 
\item if $b(z)=( z\,,\, A z)_{\bC}$ with $A\in
L(\mathbb{C}^{N};\mathbb{C}^{N})$\,, then one recovers the second
quantized version of $A$
  \begin{eqnarray*} &&b^{Wick}=
d\Gamma(A)=\sum_{j,k=1}^{N}A_{j,k}a^{*}(e_{j})a(e_{k})\,,\\
&&d\Gamma(A)\big|_{\bigvee^{n}\mathbb{C}^{N}}=
\sum_{j=0}^{n-1}I_{\mathbb{C}^{N}}^{\otimes j}\otimes A\otimes
I_{\mathbb{C}^{N}}^{\otimes n-1-j}\,;
\end{eqnarray*}
\item if $b(z)= \prod_{k=1}^{q} (
z\,,\,\zeta_{k})_{\bC} \times\prod_{j=1}^{p} (\xi_{j}\,,\,
z)_{\bC}$\,, so that the associated linear function in $\tilde{b}  \in L
\left(\bigvee^{p}\mathbb{C}^{N};\bigvee^{q}\mathbb{C}^{N} \right)$ is
$$
\tilde{b}=\mathcal{S}_{q}(|\zeta_{1} \rangle \otimes\cdots\otimes|\zeta_{q}\rangle)\otimes
\mathcal{S}_{p}( \langle \xi_{1}|\otimes\cdots\otimes \langle\xi_{p}|)\,,
$$
and then
\begin{equation}
  \label{eq:caspart} b^{Wick}=a^{*}(\zeta_{1})\ldots
a^{*}(\zeta_{q})a(\xi_{1})\ldots a(\xi_{p})\,.
\end{equation}
\end{itemize} For any polynomial $b\in
\mathbb{C}[\overline{z^{1}},\ldots, \overline{z^{N}},z^{1},\ldots,
z^{N}]$\,, $k\in\mathbb{N}$ and any $z\in \mathbb{C}^{N}$, the $k$-th
order differential $(\partial_{z}^{k}b)(z)$ is a $\mathbb{C}$-linear
form on $\bigvee^{k}\mathbb{C}^{N}$ while
$(\partial^{k}_{\overline{z}}b)(z)$ is a $\mathbb{C}$-antilinear form
which can be identified with a vector via the complex scalar product
$$
\partial^{k}_{\overline{z}}b(z): u\in\bigvee^{k}\mathbb{C}^{N}\mapsto
( u\,,\partial^{k}_{\overline{z}}b(z))_{\bC}\,.
$$
We use the notation $\ell.v$ for the $\mathbb{C}$-bilinear duality
product between $\ell\in (\bigvee^{k}\mathbb{C}^{N})^{*_{\mathbb{C}}}$
and $v\in\bigvee^{k}\mathbb{C}^{N}$\,. For any $b_{1},b_{2}\in
\mathbb{C}[\overline{z^{1}},\ldots, \overline{z^{N}},z^{1},\ldots,
z^{N}]$, $k\in\mathbb{N}$ and all $z\in \mathbb{C}^{N}$\,, the
quantity
$\partial_{z}^{k}b_{1}(z).\partial^{k}_{\overline{z}}b_{2}(z)$ is well
defined in $\mathbb{C}$ and this defines a new polynomial
$$
\partial_{z}^{k}b_{1}.\partial^{k}_{\overline{z}}b_{2}\in
\mathbb{C}[\overline{z^{1}},\ldots, \overline{z^{N}},z^{1},\ldots,
z^{N}]\,.
$$
\begin{proposition}
\label{prop:Wickcalcul}
\begin{enumerate}
\item For any monomial $b\in \mathbb{C}_{p,p}[\overline{z^{1}},\ldots,
\overline{z^{N}},z^{1},\ldots, z^{N}]$ such that $\tilde{b}\geq 0$,
the Wick quantized operator $b^{Wick}$ is non negative on
$\mathcal{D}$:
$$
\forall \varphi\in \mathcal{D}\,,\quad \langle \varphi\,,\,
b^{Wick}\varphi\rangle_{L^2} \geq 0\,.
$$
\item When $b\in \mathbb{C}[\overline{z^{1}},\ldots,
\overline{z^{N}},z^{1},\ldots, z^{N}]$ the formal adjoint of
$b^{Wick}$ defined on $\mathcal{D}$ by
$$
\forall \varphi,\psi\in \mathcal{D}\,,\quad \langle\varphi\,,\,
(b^{Wick})'\psi\rangle_{L^2}=\langle b^{Wick}\varphi\,,\, \psi\rangle_{L^2}
$$
is given by $(b^{Wick})^{'}=(\overline{b})^{Wick}$\,.
\item The set of polynomials $\mathbb{C}[\overline{z^{1}},\ldots,
\overline{z^{N}},z^{1},\ldots, z^{N}]$ is an algebra for the operation
\begin{eqnarray}
\nonumber&&(b_{1}\sharp^{Wick}b_{2})^{Wick}=b_{1}^{Wick}\circ
b_{2}^{Wick}:\mathcal{D}\to \mathcal{D} \\
\label{eq:Wickprod} \text{with}&&
b_{1}\sharp^{Wick}b_{2}=\sum_{k=0}^{\infty}\frac{1}{k!}\partial_{z}^{k}b_{1}.\partial_{\overline{z}}^{k}b_{2}\,,
\end{eqnarray} where the sum in the right-hand side is actually
finite.
\end{enumerate}
\end{proposition}
\begin{proof} 1) It comes from the definition \eqref{eq:defWick} with
$p=q$\,. Actually for any $n\in\mathbb{N}$, $\tilde{b}\otimes
I_{\bigvee^{n}\mathbb{C}^{N}}$ is non negative and
$$
\mathcal{S}_{n}\left(\tilde{b}\otimes
I_{\bigvee^{n}\mathbb{C}^{N}}\right)\mathcal{S}_{n}=
\mathcal{S}_{n}^{*}\left(\tilde{b}\otimes
I_{\bigvee^{n}\mathbb{C}^{N}}\right)\mathcal{S}_{n} \geq 0\,.
$$
For $\varphi=\oplus_{n=0}^{n_{max}}\varphi_{n}\in \mathcal{D}$ we get
$$
\left\langle\varphi\,,\,
b^{Wick}\varphi \right\rangle_{L^2}=\sum_{n=0}^{n_{max}}\left\langle\varphi_{n}\,,\,
b^{Wick}\big|_{\bigvee^{n}\mathbb{C}^{N}}\varphi_{n}\right\rangle_{L^2}\geq 0\,.
$$
2) It results from the definition \eqref{eq:defWick} after noticing
$$
\overline{b(z)}=\overline{( z^{\otimes q}\,,\,
\tilde{b}z^{\otimes p})_{\bC}}= ( z^{\otimes p}\,,\,
\tilde{b}^{*}z^{\otimes q})_{\bC},
$$
where $\tilde{b}^{*}$ is the adjoint of $\tilde{b}\in
\mathcal{L}(\bigvee^{p}\bC^{N};\bigvee^{p}\bC^{N})$\,.\\
3)
The definition~\eqref{eq:defWick} ensures that for $b\in
\mathbb{C}[\overline{z^{1}},\ldots, \overline{z^{N}},z^{1},\ldots,
z^{N}]$\,, the operator $b^{Wick}$ sends $\mathcal{D}$ into itself, so
that $b_{1}^{Wick}\circ b_{2}^{Wick}$ is well defined. By linearity
the result comes from considering the specific case when
$b_{1}^{Wick}$ and $b_{2}^{Wick}$ have the form
\eqref{eq:caspart}. With the polarization identities
\begin{eqnarray*}
\mathcal{S}_{q}(|\zeta_{1}\rangle\otimes\cdots\otimes
|\zeta_{q}\rangle)=\frac{1}{2^{q}q!}\sum_{\varepsilon_{j}=\pm
1}
\varepsilon_{1}\cdots \varepsilon_{q}
\left(\sum_{j=1}^{q}\varepsilon_{j}|\zeta_{j}\rangle\right)^{\otimes
q}\,,\\ \mathcal{S}_{p}(\langle
\xi_{1}|\otimes\cdots\otimes\langle\xi_{p}|)=
\frac{1}{2^{p}p!}\sum_{\varepsilon_{j}=\pm
1}
\varepsilon_{1}\cdots \varepsilon_{p}
\left(\sum_{j=1}^{p}\varepsilon_{j}\langle\xi_{j}|\right)^{\otimes
p}\,,\\
\end{eqnarray*} the problem is reduced to
$b_{i}^{Wick}=[a^{*}(\zeta^{i})]^{q_{i}}[a(\xi^{i})]^{p_{i}}$ for
$i=1,2$\,. But this is a simple iterated application of
$\left[a(\xi^{1}),a^{*}(\zeta^{2})\right]=(\xi^{1}\,,\,
\zeta^{2})_{\bC}$.
\end{proof}

A useful consequence of the properties stated above for our case  is the
following lemma. It provides a lower bound for a differential operator
with a specific quartic symbol. Of course, it is in a very specific
case but it is much stronger than what would give the Feffermann-Phong
inequality. Therefore, it is probably not easily accessible via the
Weyl or anti-Wick calculus (see \cite{HorIII,Ler}).
\begin{lemma}\label{lem:example}
Let $S,Q,\tilde{J}$ be real matrices such that $S \in {\mathcal S}_N^{>0}(\R)$,
$Q \in {\mathcal S}_N^{>0}(\R)$ and $\tilde{J} \in {\mathcal A}_N(\R)$. Let us consider the
operator $L=-\tilde{\mathcal L}_J$ (see Equation~\ref{eq:deftLJ} for
the definition) and
$C=C_Q$ (see Equation~\eqref{eq:CQ}). The operator $L$ (respectively $C$) is
the Wick quantization of the polynomial $\ell (z)=( z\,,\, Sz)_{\bC} - (
z\,,\,\tilde{J}z)_{\bC}$ (respectively $p(z)=( z\,,\, Q
z)_{\bC}$). Moreover, we have the following estimate: $\forall \varphi\in \mathcal{D}$,
\begin{equation}
  \label{eq:ineqexample}
\langle \varphi\,,\,
(L^{*}C+CL)\varphi\rangle_{L^2}
\geq 
\left\langle \varphi\,,\, \left( \left( z \,,\,
\left[SQ+QS+\tilde{J}Q-Q\tilde{J}\right]z \right)_{\bC}\right)^{Wick}\varphi \right\rangle_{L^2}.
\end{equation}
\end{lemma}
\begin{proof}
The fact that
$L=(\ell(z))^{Wick}$ and  $C=(p(z))^{Wick}$
is easy to check. Notice that the polynomials $\ell$ and $p$ satisfy
$$
{\rm Re}~\ell(z)=( z\,,\, Sz)_{\bC}\,,\quad
{\rm Im}~\ell(z)=-\frac{1}{i}(z\,, \tilde{J} z)_{\bC}\,,\quad
\overline{p(z)}=p(z)\,.
$$
We are looking for a lower bound for $L^{*}C+CL$. Using
formula~\eqref{eq:Wickprod}, the Wick symbol of $L^{*}C+CL$ is
\begin{multline*} \overline{\ell(z)} \, p(z)+p(z) \, \ell(z)
+\partial_{z}\overline{\ell(z)} \,. \, \partial_{\overline
z}p(z)+ \partial_{z}p(z) \,. \, \partial_{\overline z}\ell(z) \\ 
=\left(
z^{\otimes 2}\,,\, (S\otimes Q+Q\otimes S)z^{\otimes 2} \right)_{\bC} +
\left( z\,,\, \left(SQ+QS+\tilde{J}Q-Q\tilde{J}\right)z \right)_{\bC}.
\end{multline*}
Since $S$ and $Q$ are non negative matrices, we deduce
that $S\otimes Q$ and $Q\otimes S$ are non negative and the first term
is thus non negative. By applying the first statement of Proposition~\ref{prop:Wickcalcul},
one obtains~\eqref{eq:ineqexample}.
\end{proof}

\begin{acknowledgements}
We would like to thank Matthieu Dubois for preliminary numerical experiments.
\end{acknowledgements}



\end{document}